\documentclass[reqno]{amsart}
\usepackage{kantlipsum} % for text filler
\calclayout
\usepackage[utf8]{inputenc}
\usepackage{amssymb,amsthm,amsfonts, amsmath}
\usepackage{tikz}
\usepackage{graphicx}
\usepackage{psfrag}
\usepackage{cancel}
\usepackage{hyperref}
\usepackage{mathrsfs}
\usepackage{microtype}
\usepackage{mathtools}
\usepackage{multicol}
\usepackage{lipsum}
\usepackage{caption}
\usepackage{subcaption}
\usepackage{stfloats} 

\newtheorem{definition}{Definition}[section]
\newtheorem{theorem}{Theorem}[section]
\newtheorem{lemma}[theorem]{Lemma}
\newtheorem{corollary}[theorem]{Corollary}
\newtheorem{proposition}{Proposition}[section]
\newtheorem{remark}{Remark}[section]

\newtheorem{claim}{Claim}[section]

\theoremstyle{plain}
\newtheorem{maintheorem}{Theorem}

\newcommand{\nt}{\ensuremath{\mathbb{N}}}
\newcommand{\R}{\ensuremath{\mathbb{R}}}

\newcommand{\Q}{\ensuremath{\mathbb{Q}}}
\newcommand{\Z}{\ensuremath{\mathbb{Z}}}
\newcommand{\crit}{\operatorname{Crit}}

\title[$C^1$ robust rigidity for bi-critical circle maps]{$C^1$ robust rigidity for bi-critical circle maps}

\author{Gabriela Estevez}
\address{Instituto de Matem\'atica e Estat\'istica, Universidade Federal Fluminense}
\curraddr{Rua Prof. Marcos Waldemar de Freitas Reis, S/N, 24.210-201, Bloco G, Niter\'oi, Brasil}
\email{gestevez@id.uff.br}



\thanks{The author was partially financed by FAPERJ and CNPQ (Chamada Universal 10/2023 faixa B $404943/2023-3$)}

\begin{document}

	\newcounter{casenum}
	\newenvironment{caseof}{\setcounter{casenum}{1}}{\vskip.5\baselineskip}
	\newcommand{\case}[2]{\vskip.5\baselineskip\par\noindent {\bfseries Case \arabic{casenum}:} #1\\#2\addtocounter{casenum}{1}}
	
	\begin{abstract} We prove that  two topologically conjugate  bi-critical circle maps whose signatures are the same, and whose renormalizations converge together exponentially fast in the $C^2$-topology, are $C^1$ conjugate. 
	\end{abstract}
	
	\maketitle

	\section{Introduction}\label{sec:introduction}
	In one-dimensional dynamics rigidity occurs when the existence of a weak equivalence between two systems implies the existence of a stronger equivalence between the systems. The general rigidity problem that we are interested in is: when does the existence of a topological conjugacy between two circle maps imply the existence of a smooth conjugacy  between these maps?
	
	The study of smoothness of conjugacies between circle maps has been studied at least since the 1960's.
	Smooth circle diffeomorphisms are topologically conjugate to a rigid rotation, and the smoothness of this conjugacy is now well understood. 
	 Herman \cite{H79} and Yoccoz \cite{yoccoz2} proved that if $f$ is a $C^r$ diffeomorphism of the circle,
	whose rotation number $\rho$ satisfies $	 |\rho-p/q | \geq C \, q^{-2-\beta}$ for all $p/q \in \mathbb{Q}$, for some $C > 0$ and $0 \leq \beta<1$, then, provided $r>2\beta+1$, $f$ is $C^{r-1-\beta-\varepsilon}$-conjugate to the corresponding rigid rotation for every $\varepsilon >0$. These results were also proved by Stark \cite{St} and Khanin-Sinai \cite{KS} using  renormalization methods. Later, Khanin and Teplinsky \cite{KT09} showed that the result is also valid for $\varepsilon=0$, provided that $2<r<3$. 
	
	Rigidity results have also been shown for classes of circle maps with critical points, in particular for multi-critical circle maps.
	Multi-critical circle maps are $C^3$ maps of the circle with a finite number of non-flat critical points. Any non-flat critical point has an associated real number which is called the criticality of the critical point.  Yoccoz proved in \cite{yoccoz} that any multi-critical circle map is topologically conjugate to the corresponding rigid rotation, in particular  two multi-critical circle maps  with the same irrational rotation number are topologically conjugate to each other. 
	We are interested in studying the smoothness of the conjugacy between two critical circle maps as opposed to the conjugacy between a critical circle map and the corresponding rigid rotation. 
	If each of the maps only have one critical point and the criticalities of the two critical points are the same then the following smoothness results have been shown for the conjugacy $h$ that sends one critical point to the other: $h$ is quasi-symmetric (and therefore Hölder) \cite{dFdM}; $h$ is a $C^1$ diffeomorphism \cite{KT, GMdM} and it is a $C^{1+\alpha}$ diffeomorphism for a full Lebesgue measure subset $\mathcal{A}$ of irrational rotation numbers \cite{dFdM}, provided that the critical maps are $C^4$. Moreover, it has been shown \cite{Av,dFdM} that there exist real-analytic and $C^{\infty}$ uni-critical circle maps whose rotation numbers are not contained in $\mathcal{A}$ where the conjugacy is only $C^1$ and not $C^{1 +\alpha}$. 
	
	If the two maps have more that one critical point
	additional topological invariants appear. Let $f$ be a $C^3$ multi-critical circle map with irrational rotation number $\rho$, let $\mu$ be the unique $f-$invariant probability measure (since $f$ is topologically conjugate to an irrational rotation then $f$ is uniquely ergodic), and let  $c_0, \dots, c_{N-1}$ be the non-flat critical points with criticalities  $d_0, \dots, d_{N-1}$ respectively (see Section \ref{sec:preliminares}). The  \emph{signature} of the multi-critical circle map $f$ is defined to be the $(2N+2)$-tuple
	\[
	(\rho\,;N;\,d_0,d_1,\ldots,d_{N-1};\,\delta_0,\delta_1,\ldots,\delta_{N-1}),
	\]
	where  $\delta_i=\mu[c_i,c_{i+1})$ (with  $c_{N}=c_0$).
	
	It is a straight forward to see that if $f$ and $g$ are two $C^3$ multi-critical circle maps with the same signature then there is a conjugacy $h$ between them  that identifies each critical point of $f$ with one of $g$ and preserves the  criticalities. % We say that $c^f \in Crit(f)$ and $c^g \in Crit(g)$ are \textit{corresponding critical points} if there is a conjugacy between $f$ and $g$ that sends $c^f$ to $c^g$, preserving the criticalities.
	In this paper, we study bi-critical circle maps, these are multi-critical circle maps that have precisely two non-flat critical points. We show  the existence of a $C^{1}$ conjugacy between bi-critical circle maps (see Theorem \ref{maintheorem} below) without any restriction on the irrational rotation number. Our main result is the following:	
	\begin{maintheorem}\label{maintheorem} Let $f$ and $g$ be two $C^3$ bi-critical circle maps with the same signature. If the renormalizations of $f$ and $g$ around corresponding critical points converge together exponentially fast in the $C^2$ topology, then $f$ and $g$ are $C^1$-conjugated.
	\end{maintheorem}

    In \cite{EG}, it was shown that two $C^3$ bi-critical circle maps $f,g$ are $C^{1+\alpha}$-conjugate if their rotation number belongs to a total Lebesgue measure set $\mathcal{A}$, provided that their renormalizations converge together in the $C^1$ topology (instead of the $C^2$ topology required in Theorem \ref{maintheorem}). In particular those bi-critical circle maps are $C^1$-conjugate. The condition on the rotation number belonging to $\mathcal{A}$ let the authors avoid the use of the $C^2$ topology. In this paper we are not imposing any condition on the irrational rotation number, due to this we have to deal with the existence of almost tangency points as well as their first two derivatives under renormalization (see Section 5). Assuming that the renormalizations of $f$ and $g$ converge together exponentially in the $C^2$ topology allows us to obtain a suitable control on the first two derivatives of tangency points of $f$ and $g$.

   % Necessary properties on the first two derivatives at almost tangency points are guaranteed by assuming that the renormalizations of $f$ and $g$ converge together exponentially in the $C^2$ topology.  
    
     % In the proof of Theorem \ref{maintheorem}, we will not just consider  the case where the rotation number does not belongs to the set $\mathcal{A}$. Instead, we will give arguments for the cases when the rotation number is bounded type and when it is unbounded type. Since bounded type rotation numbers are contained in $\mathcal{A}$, then we will focus in the case of un-bounded type rotation numbers, see Section \ref{subsec:funnel}.

	\section{Preliminaries}\label{sec:preliminares}
	
	\subsection{Multicritical circle maps}\label{subseccomb} A \emph{multicritical circle map} is an orientation preserving $C^3$ circle homeomorphism having $N \geq 1$ non-flat critical points and with irrational rotation number. A critical point $c$ of a $C^3$ map $f$ is \emph{non-flat} of criticality $d>1$ if there exists a neighborhood $W$ of the critical point such that $f(x)=f(c)+\phi(x)\,\big|\phi(x)\big|^{d-1}$ for all $x \in W$, where $\phi : W \rightarrow \phi(W)$ is an orientation preserving $C^{3}$ diffeomorphism satisfying $\phi(c)=0$.
	As we mentioned in the introduction, any multicritical circle map $f$ has no wandering intervals and in particular it is topologically conjugate with the corresponding rigid rotation \cite{yoccoz}.
	
%	Let us recall  two examples of multicritical circle maps. As a first example,
 As an example, we consider a specific case of the generalized Arnold's family: Let $a \in [0,1)$, $N\in\nt$ and let  $\tilde{f}_a:\R\to\R$ be given by:$$\tilde{f}_a(x)=x+a-\frac{1}{2N \pi}\,\sin (2 N \pi x).$$
	
	Each $\tilde{f}_a$ is a lift of $f_a$, an orientation preserving real-analytic circle homeomorphism, under the  universal cover $x \mapsto e^{2\pi ix}$. We can see that each $f_a$ is a critical circle map with  $N$ cubic critical points, given by $\big\{e^{\frac{j}{N}2\pi i}:\, 0 \leq j \leq N-1 \big\}$.
	
	%As a second example, let $p,q \in \C$ with $|p| > 1$, $|q| > 1$, $t \in [0,1)$ and consider the  family $f_{p,q,t} : \C \to \C$ of Blaschke products given by
%	\begin{equation}\label{eq:Zakerifamily}
%	f_{p,q,t}(z) = e^{2\pi it}z^3 \left( \frac{z-p}{1-\overline{p}z} \right) \left(\frac{z-q}{1-\overline{q}z}\right).
%	\end{equation}
	
%	Zakeri proved that for any irrational rotation number $\rho$ and for any $\delta \in (0,1)$, there exists an unique map $f_{p,q,\rho}$ of the form \eqref{eq:Zakerifamily}, such that $f_{p,q,\rho}|_{S^1}$ is a bi-cubic circle map with signature given by $(\rho;2;3,3; \delta,1 -\delta)$, see \cite{zak}.
	
	\subsection{Combinatorics of multicritical circle maps}
	For an irrational rotation number $\rho \in [0,1)$, let us consider its infinite continued fraction expansion:
\begin{equation*}
	\rho= [a_{0} , a_{1} , \cdots ]=   
	\cfrac{1}{a_{0}+\cfrac{1}{a_{1}+\cfrac{1}{ \ddots} }} \ .
\end{equation*}

 %\,:$$\frac{p_n}{q_n}\;=\;[a_0,a_1, \cdots ,a_{n-1}]\;=\;\cfrac{1}{a_1+\cfrac{1}{\ddots\cfrac{1}{a_{n-1}}}}\ .$$

Truncating the continued fraction expansion, we obtain the  \textit{convergents} of $\rho$, given by $\frac{p_n}{q_n}=\;[a_0,a_1, \cdots ,a_{n-1}].$
The sequence of denominators $\{q_n\}_{n\in\nt}$ satisfies the following recursive formula:
\begin{equation*}
	q_{0}=1, \hspace{0.4cm} q_{1}=a_{0}, \hspace{0.4cm} q_{n+1}=a_{n}\,q_{n}+q_{n-1} \hspace{0.3cm} \text{for all $n \geq 1$} .
\end{equation*}

Given  a preserving orientation circle homeomorphism $f$ with irrational rotation number $\rho$ and given any $x \in S^1$, the sequence of iterates $\{f^{q_n}(x)\}_{n \in\nt}$ converges to $x$ (but never coincides with $x$ since $f$ has no periodic points) alternating the sides where it approaches. 
Moreover, the iterate $f^{q_n+q_{n+1}}(x)$ is closer to $x$ than $f^{q_n}(x)$ but further to $x$ than $f^{q_{n+1}}(x)$. By orientation, $f^{q_n+q_{n+1}}(x)$ must be inside the smallest arc that connects $x$ with $f^{q_n}(x)$. We denote by $I_{n}(x)$ the interval with endpoints $x$ and $f^{q_n}(x)$, which contains the point $f^{q_{n+2}}(x)$. The collection of intervals
\[
\mathcal{P}_n(x) \ = \ \left\{ f^{i}(I_n(x)):\;0\leq i\leq q_{n+1}-1 \right\} \;\bigcup\; 
\left\{ f^{j}(I_{n+1}(x)):\;0\leq j\leq q_{n}-1 \right\} 
\]
is a partition of the circle (modulo endpoints) called the {\it $n$-th (classical) dynamical partition\/} associated to $x$. The sequence of dynamical partitions $\{\mathcal{P}_n(x)\}_{n \in \nt}$, for any $x \in S^1$, is \textit{nested}: each element of $\mathcal{P}_{n+1}(x)$ is contained in an element of 	$\mathcal{P}_n(x)$. Those partitions are also \textit{refining}: the maximal length of an element of 	$\mathcal{P}_n(x)$ converges to zero. For any $x \in S^1$, we denote by $J_n(x)$ the union $I_n(x)\cup I_{n+1}(x)$ and  for any $n \in \nt$, we denote by $E_{n}$ the set of endpoints of $\mathcal{P}_n(c)$ for $c \in Crit(f)$. With dynamical partitions at hand, we are able to define  the renormalization of a critical circle map.

	\subsection{Renormalization of bi-critical circle maps}\label{rmcm}
	We recall the notion of  \textit{bi-critical commuting pairs}, which is a natural generalization of the  classical notion of critical commuting pairs.
	
	\begin{definition}\label{multcommpairs} A $C^3$ critical commuting pair with two %$N=N_1+N_2-1$ 
	critical points, or a bi-critical commuting pair, is a pair $\zeta=(\eta,\xi)$ consisting of two $C^3$ orientation preserving homeomorphisms $\xi:I_\xi \rightarrow  \xi(I_\xi)$ and $\eta : I_\eta \rightarrow \eta(I_{\eta})$% with one or none %a finite number of non-flat 		critical points $0$ and $\beta$, respectively
		, satisfying:
		\begin{enumerate}
			\item $I_{\xi}=[\eta(0), 0]$ and $I_{\eta}=[0, \xi(0)]$ are compact intervals in the real line;
			\item $(\eta \circ \xi)(0)=(\xi \circ \eta)(0) \neq 0$;
			\item $D\xi(x)>0$ for all $x$ %$\gamma_i < x< \gamma_{i+1}$, $i \in \{0, 1, \cdots, N_1 -1\}$ 
			and $D\eta(x)>0$ for all $x \neq \beta$;
		%	$\beta_j < x <\beta_{j+1}-1$,  $j \in \{ 0, 1, \cdots, N_2-1 \}$;
			\item The origin has the same criticality for $\eta$ than for $\xi$;
			\item For each $1 \leq k \leq 3$, we have that $D^{k}_{-}(\xi \circ \eta)(0)=D^{k}_{+}(\eta \circ \xi)(0)$, where $D_{-}^k$ and $D_{+}^k$ represent the $k$-th left and right derivative, respectively.
		\end{enumerate}
	\end{definition}

%	\begin{subfigure}{0.5\linewidth}
%		\psfrag{n}[][]{$\eta$} 
%		\psfrag{e}[][]{$\xi$}
%		\psfrag{0}[][]{$0$}
%		\includegraphics[width=2.5in]{mcp.eps}
%		\caption{A bi-critical commuting pair $\zeta=(\eta,\xi)$.}
%	\end{subfigure}

	Let $\zeta_1=(\eta_{1}, \xi_{1})$ and $\zeta_2=(\eta_2, \xi_2)$ be two $C^r$ bi-critical %multicritical
	 commuting pairs, and let
	$\tau_1: [\eta_1(0), \xi_1(0)] \rightarrow [-1,1]$ and $\tau_2: [\eta_2(0), \xi_2(0)] \rightarrow [-1,1]$ be the two M\oe{}bius transformations
	with
	$$ \tau_{i}(\eta_{i}(0))=-1, \ \ \tau_{i}(0)=0 \ \text{ and } \ \tau_{i}(\xi_{i}(0))=1,$$
	for $i \in \{1,2\}$. For all $0 \leq r < \infty$ we define the $C^r$ pseudo-distance between $\zeta_1$ and $\zeta_2$ as:
		\begin{equation*}
			d_{r}(\zeta_1,\zeta_2)= \max \left\lbrace \left| \dfrac{\xi_1(0)}{\eta_1(0)} - \dfrac{\xi_2(0)}{\eta_2(0)} \right| ,  \
			\| \tau_{1}\circ \zeta_1 \circ \tau_1^{-1} - \tau_2 \circ \zeta_2 \circ \tau_{2}^{-1} \|_{r} \right\rbrace
		\end{equation*}
		where $\| \cdot \|_{r}$ represents the  $C^{r}$ norm for maps in the interval $[-1,1]$ with a discontinuity at the origin. To get a distance, we restrict to \emph{normalized} bi-critical %multicritical 
	commuting pairs: for any given pair $\zeta=(\eta,\xi)$ we denote by $\widetilde{\zeta}$ the pair $(\widetilde{\eta}|_{\widetilde{I_{\eta}}}, \widetilde{\xi}|_{\widetilde{I_{\xi}}})$, where tilde means linear rescaling by the factor $1/|I_{\xi}|$. Note that $|\widetilde{I_{\xi}}|=1$ and $|\widetilde{I_{\eta}}| =|I_{\eta}|/|I_{\xi}|$. Equivalently $\widetilde{\eta}(0)=-1$ and $\widetilde{\xi}(0)=|I_{\eta}|/|I_{\xi}|=\xi(0)/\big|\eta(0)\big|$.
	
    The \emph{period} of the pair $\zeta=(\eta, \xi)$ is the number $\chi(\zeta) \in \nt$ such that $ \eta^{\chi(\zeta)}(\xi(0)) <0\leq \eta^{\chi(\zeta)-1}(\xi(0)),$ when such number exists. If such number does not exist, {\it i.e.\/} when $\eta$ has a fixed point, we write $\chi(\zeta)=\infty$.

	\begin{definition}\label{defren} For a pair  $\zeta=(\eta, \xi)$ with $\chi(\zeta)< \infty$ and  $(\xi \circ \eta)(0) \in I_{\eta}$, we define its \emph{pre-renormalization} as being the pair$$p\mathcal{R}(\zeta) =\left(\eta|_{[0, \eta^{\chi(\zeta)}(\xi(0)) ]} \ , \ \eta^{\chi(\zeta)}\circ \xi|_{I_{\xi}} \right).$$
	Moreover, we define the \emph{renormalization} of $\zeta$ as the  rescaling of $p\mathcal{R}(\zeta)$, that is,
	\begin{align*}
		%\mathcal{R}\colon \left[-\frac{|I_{\xi}|}{|I_{\eta}|},1\right]  & \longrightarrow \left[-\frac{|I_{\xi}|}{|I_{\eta}|},1 \right]\\
		%\zeta&\longmapsto
		 \mathcal{R}(\zeta)=  \left(\widetilde{\eta}|_{[0,\widetilde{\eta^{\chi(\zeta)}(\xi(0))} ]} \ , \ \widetilde{\eta^{\chi(\zeta)}\circ \xi}|_{\widetilde{I}_{\xi}}
		\right).
	\end{align*}
	\end{definition}
	
	If $\zeta$ is a multicritical commuting pair with $\chi(\mathcal{R}^{j}\zeta)< \infty$ for $0 \leq j \leq n-1$, we say that $\zeta$ is 	\textit{$n$-times renormalizable}, and if $\chi(\mathcal{R}^{j}\zeta)< \infty$  for all $j \in \nt$, we say that $\zeta$ is \textit{infinitely renormalizable}. In the last case, we define the \textit{rotation number} of the bi-critical commuting pair $\zeta$ as the irrational number whose continued fraction expansion is given by $[\chi(\zeta), \chi(\mathcal{R}\zeta), \cdots, \chi(\mathcal{R}^{n}\zeta), \cdots ].$

	Let us recall how we obtain a bi-critical % multicritical
	 commuting pair from a given  bi-critical %multicritical
	  circle map. Let $f$ be a $C^r$ bi-critical circle map with $\rho(f)=\rho \in \R\setminus \Q$ and  critical points $c_{0}, c_{1}$. For a given $c_i$, let $\widehat{f}$ be the lift of $f$ (under the universal covering $t \mapsto c_i\cdot\exp(2\pi i t)$) such that $0< \widehat{f}(0)<1$ (note that $D\widehat{f}(0)=0$). For $n\geq 1$, let $\widehat{I}_{n}(c_i)$ be the closed interval in $\R$, that has the origin as one of its endpoints and which projects onto $I_{n}(c_i)$. We define $\xi: \widehat{I}_{n+1}(c_i) \rightarrow \R$ and $\eta: \widehat{I}_{n}(c_i) \rightarrow \R$ by $\xi= T^{-p_{n}}\circ \widehat{f}^{q_{n}}$ and $\eta= T^{-p_{n+1}}\circ \widehat{f}^{q_{n+1}}$, where  $T(x)=x+1$. Then the pair $(\eta|_{\widehat{I}_{n}(c_i)}, \xi|_{\widehat{I}_{n+1}(c_i)})$ is a renormalizable bi-critical commuting pair, and its rescaling is denoted by $\mathcal{R}_i^{n}f=\left(f^{q_{n+1}}|_{I_n(c_i)}, f^{q_n}|_{I_{n+1}(c_i)}\right).$
In other words, let us consider the $n-$th scaling ratio of $f$:$$s_{n,i}^f=\frac{|I_{n+1}(c_i)|}{|I_n(c_i)|},  \ \ \ \ \text{for $i \in \{0,1\}$.}$$ Then the $n-$th renormalization of $f$ at $c_i$ is the commuting pair 
	$\mathcal{R}_i^nf : [-s_{n,i}^f, 1] \to [-s_{n,i}^f, 1]$ given by
	\[ \mathcal{R}_i^nf=
	\begin{dcases}
		B_{n,f}\circ f^{q_n} \circ B_{n,f}^{-1} & \mbox{in  $[-s_{n,i}^f,0)$}\\[1ex]
		B_{n,f}\circ f^{q_{n+1}} \circ B_{n,f}^{-1} & \mbox{in $[0,1]$,}\\
	\end{dcases}
	\]
	where $B_{n,f}$ is the unique orientation-preserving affine diffeomorphism between $J_n(c_i)=I_{n+1}(c_i)\cup I_{n}(c_i)$ and $[-s_{n,i}^f,1]$, that is,
	\[
	B_{n,f}(t)= (-1)^n \frac{|t-c_i|}{|I_n(c_i)|}.
	\]
	
	As a straightforward consequence of the combinatorics, we know that $f$ has irrational rotation number if and only if $f$ is infinitely renormalizable. In fact, once we have the $n-$th renormalization of a pair we can  obtain the  $(n+1)-$th renormalization by iterating $\mathcal{R}_i^{n}f(B_{n,f}(1))$ as many times until we observe a change in the sign of the images. By combinatorics that number coincides with $ \ a_{n+1}+1 $, since$$(\mathcal{R}_i^{n} f)^{a_{n+1}+1}(B_{n,f}(1)) <0 <(\mathcal{R}_i^nf)^{a_{n+1}}(B_{n,f}(1)).$$ The $(n+1)-$th renormalization (whose graph before being normalized is represented by the pair of maps inside the small blue box in Figure \ref{fig:sequencerenor}) is$$\left(B_{n+1,f}\circ f^{q_{n+2}} \circ B_{n+1,f}^{-1} |_{[-1,0]}, \ B_{n+1,f}\circ f^{q_{n+1}} \circ B_{n+1,f}^{-1} |_{[0.s_{n+1,i}^f]}\right).$$

	\begin{figure}
		\includegraphics[width=3.6in]{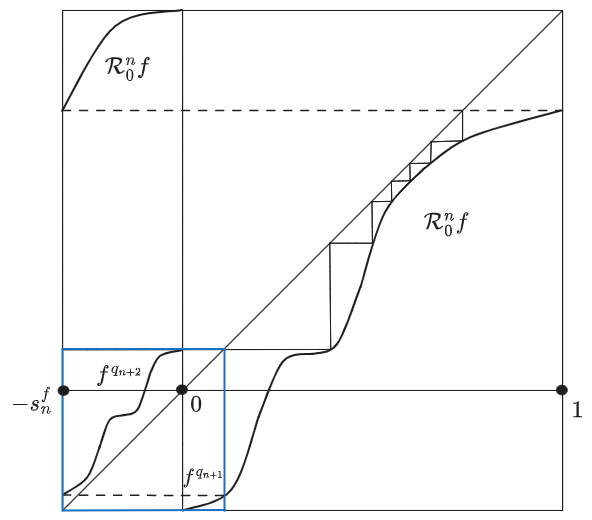}
		\caption{\label{fig:sequencerenor} The $n-$th  renormalization (in black) and the $(n+1)$-th pre-renormalization (in blue) of a bi-critical circle map $f$ around the critical point $c_0$.}
	\end{figure}
	
	%\begin{figure}
	%	\centering
	%	\psfrag{F}[][]{$\widetilde{f}^{q_{n+1}}$}
	%	\psfrag{G}[][]{$ \ \ \ \widetilde{f}^{q_n}$}
	%	\psfrag{H}[][]{$ \ \ \widetilde{f}^{q_{n+2}}$}
	%	\includegraphics[width=2.8in]{renormalizationofmcm.eps}
	%	\caption{\label{fig:sequencerenor} The $n-$th and $(n+1)-$th renormalization of a bi-critical circle map (without rescalling) around the critical point $c_i$.}
	%\end{figure} 

    Let $f$ and $g$ be two critical circle maps  with same signature. We say that their corresponding renormalizations converge together exponentially fast, around the critical point $c_i$, in the $C^r$-topology if $f$ and $g$ are infinitely renormalizable and there exist $n_0 \in \nt$, $C>0$, $\lambda \in (0,1)$ (both depending on the signature) such that for all $n \geq n_0$
    \begin{equation} \label{ineq:convergencerenormalizations}
    d_r(\mathcal{R}_i^nf, \mathcal{R}_i^ng) \leq C\, \lambda^n.
    \end{equation}
    In particular, inequality \eqref{ineq:convergencerenormalizations} implies that  for all $n \geq n_0$
    \begin{equation*} 
    \left| \frac{|I_{n}(c_i)|}{|I_{n+1}(c_i)|} - \frac{|h(I_{n}(c_i))|}{|h(I_{n+1}(c_i))|}  \right| \leq C\, \lambda^n.    
    \end{equation*}

		\subsection{Real bounds}
	Now, let us mention  a geometric control on the dynamical partitions that implies the pre-compactness of the set $\{\mathcal{R}^nf\}_{n\in \nt}$ in the $C^0$-topology. That geometric control is called real bounds (see  \cite{H88,Swi} and \cite{EdF,EdFG}).
	
	\begin{theorem}[Universal Real Bounds]\label{teobeau} Given $N \in \nt$ and $d \in \R$, with $d>1$, let $\mathcal{F}_{d}$ be the family of $C^3$ bi-critical circle maps % with at most $N$ critical points 
		whose maximum criticality is bounded by $d$. There exists $C_0=C_0(N,d)>1$ such that for any $f \in \mathcal{F}_{d}$ and $c \in \crit(f)$ there exists $n_0 \in \mathbb{N}$ with the property that: for all $n\geq n_0$ and every adjacent intervals $I,J \in \mathcal{P}_{n}(c)$, we have$$ \frac{1}{C_0} \leq \frac{|I|}{|J|} \leq C_0\,.$$
	\end{theorem}
 
    In other words, real bounds say that the lengths of adjacent intervals in the dynamical partition are comparable. Real bounds are also true even if the circle map has not only two but a finite number of critical points (see  \cite{EdF,EdFG}). \\
    
    The following result is a consequence of having universal real bounds.
	
	\begin{corollary}\label{corbeau}(Exponential refining) There exists a constant $\mu=\mu(C_0) \in (0,1)$, universal in $\mathcal{F}_{N,d}$, such that for all $m>n\geq n_0$: if $\mathcal{P}_{m}(c) \ni I \subseteq J \in \mathcal{P}_{n}(c)$, then $|I| \leq \mu^{m-n} |J|$. 
	\end{corollary}

	%\begin{corollary}\cite[Lemma 3.1]{EdFG} \label{c1bounds} Let $f \in  \mathcal{F}_{N,d}$ and $c \in \crit(f)$. There exists $K_1=K_1(f)>1$ such that for all $n \geq n_0$ and $x \in J_n(c)$, we have $Df^{q_{n+1}}(x) \leq K_1$ (for $x \in I_n(c)$)  and  $Df^{q_{n}}(x) \leq K_1$ (for $x \in I_{n+1}(c)$).
	%\end{corollary}

\subsection{Two tools: Schwarzian derivative and Koebe Distortion Principle}

Recall that for a given $C^3$ map $f$ we can define its Schwarzian derivative, which is the differential operator defined (for all $x$ regular point of $f$) by:
\[
S f (x) = \frac{D^3 f (x)}{Df (x)} - \frac{3}{2} \left(
\frac{D^2 f (x)}{Df (x)} \right)^2.
\]

For multicritical circle maps we have the following result concerning the Schwarzian derivative:

\begin{lemma}\cite[Lemma 4.1]{EdFG} \label{lemma_schwarzianderivative} Let $f$ be a multicritical circle map and $c \in \crit(f)$. There exists $n_1= n_1(f ) \in \nt$ such that for all $n \geq n_1$ we have that:
	\[
	S f^j (x) < 0 \text{ \ for all $j \in \{1, \dots , q_{n+1}\}$ and for all $x \in I_n(c)$ regular point of $f^j$,}
	\]
	and
	\[
	S f^j (x) < 0 \text{ \ for all $j \in \{1, \dots, q_n\}$ and for all $x \in I_{n+1}(c)$ regular point of $f^j$.}
	\]
\end{lemma}

It is known that the  Schwarzian derivative of the composition of two (sufficiently differentiable) functions $f$ and $g$ is given by:
\[
S(f\circ g)(z)=\left((Sf)\circ g\right)(z) \cdot (Dg(z))^{2}+Sg(z).
\]
Moreover,  it is also know that Mobius transformations are the only functions with zero Schwarzian derivative. Hence, the Schwarzian derivative of the renormalization is given by:% if $g$ is a Mobius transformation, then 
%\[
%S(f\circ g)(z)=((Sf) \circ g)(z)\cdot (Dg(z))^{2}.
%\]
\[ S(\mathcal{R}_0^nf)=
\begin{dcases}
	\left((Sf^{q_n})\circ B_{n,f}^{-1}\right)\cdot \left(DB_{n,f}^{-1}(z)\right)^2 & \mbox{in  $[-s_{n,0}^f,0)$}\\[1ex]
	\left((Sf^{q_{n+1}}) \circ B_{n,f}^{-1}\right)\cdot \left(DB_{n,f}^{-1}(z)\right)^2 & \mbox{in $[0,1]$.}\\
\end{dcases}
\]
In particular, $S(\mathcal{R}_0^nf)(z)<0$  for all $n\geq n_1$ and for all  $z \in [-s_{n,0}^f,1]$  regular point.

% Since the first returns $\{f^{q_n}(c_0)\}_{n \in \nt}$ are approaching to $c_0$ alterning the order, we observe that
%$\mathcal{R}^{n-1}f : [-1, s_{n-1}^f] \to [-1, s_{n-1}^f]$ is given by
%\[ \mathcal{R}^{n-1}f=
%\begin{dcases}
%	B_{n-1,f}\circ f^{q_{n-1}} \circ B_{n-1,f}^{-1} & \mbox{in  $[0,s_{n-1}^f]$}\\[1ex]
%	B_{n-1,f}\circ f^{q_{n}} \circ B_{n-1,f}^{-1} & \mbox{in $[-1,0]$,}\\
%\end{dcases}
%\]
The Koebe Distortion Principle gives a control on the distortion of high iterates of the map. Recall that given two intervals $M \subset T \subset S^1$, with $M\Subset T$, we define the \textit{space} of $M$ inside $T$ as the smallest of the ratios $|L|/|M|$ and $|R|/|M|$, where $L$ and $R$ denote the two connected components of $T \setminus M$. If the space is $\tau> 0$, we say that $T$ contains a \textit{$\tau$-scaled} neighborhood of $M$. For multicritical circle maps we have the following  Koebe Distortion Principle.

\begin{lemma} \label{KoebeDistortionPrinciple}
	For each $\gamma, \tau > 0$ and each multicritical circle map $f$, there exists a constant 
	\[
	K = \left( 1 + \frac{1}{\tau} \right)^2 \exp(\widetilde{C_0} \gamma) \ > 1,
	\]
	where $\widetilde{C_0}>0$  only depends on $f$, with the following property: if $T$ is an interval such
	that $f^k|_T$ is a diffeomorphism onto its image and if
	$\sum_{j=0}^{k-1} |f^j (T)| \leq \gamma$, then for each interval 	$M \subset T$ for which $f^k(T)$ contains a $\tau$-scaled neighborhood of $f^k(M)$ we have for all $x, y \in M$
	\[
	\frac{1}{K} \leq \frac{|Df^k (x)|}{|Df^k (y)|} \leq K.
	\]
\end{lemma}

	Next result give us a decomposition of the iterates of intervals in the dynamical partition.

\begin{lemma}\label{lemma_decompositiontwobridge}  (Decomposition) \cite[Lemma 5.1]{EdFG}
	Given $\varepsilon>0$ there exists $n_2 \in \nt$, depending on $f$ and $\varepsilon$, with the following property: given $n \geq n_2$, $I \in \mathcal{P}_{n}$ and $k \in \nt$ for which $f^{j}(I)$ is contained in an element of $\mathcal{P}_{n}$, for all $1\leq j \leq k$, we have the following decomposition:
	$$f^{k}(I^*)= \widehat{\phi}_k \circ \widehat{\phi}_{k-1} \circ \dots \circ \widehat{\phi}_1$$
	where,
	\begin{enumerate}
		\item For at most $7$ values of $i \in \{1, \cdots, k\}$, $\widehat{\phi}_i$ is a diffeomorphism with bounded distortion by $1+ \varepsilon$.
		\item For at most $6$ values of $i \in \{1, \cdots, k\}$, $\widehat{\phi}_i$ is the restriction of $f$ to some interval contained in $U_i$.
		\item For the reminder values of $i$, $\widehat{\phi}_i$ is either the identity or a diffeomorphism with negative Schwarzian derivative.
	\end{enumerate}	
\end{lemma}

\subsection{Controlled bi-critical commuting pairs}
    
    Real bounds also give us some control on the length of certain intervals, dynamically defined, and some control on the first and second derivatives of the renormalizations.   

   \begin{definition}($M$-controlled bi-critical commuting pair)\label{def:Kcontrolledcommutingpair}
    Let $M >1$ and let $\zeta=(\eta, \xi)$  be a normalized $C^3$ bi-critical commuting pair, with period $a =\chi(\zeta) \in  \nt$ and with %$N$ 
    critical points given by $0$ and $\beta \in (0,\xi(0))$. %_1, \cdots,  \beta_{N-1} \in (0,\xi(0))$. 
     We say that $\zeta$ is $M$-controlled if the following conditions are satisfied:
    \begin{enumerate}
    	\item $1/M \leq \xi(0) \leq M$; 
    	\item $1/M \leq \xi(0)-\eta(\xi(0))$;
    	\item $1/M \leq \eta^{a-1}(\xi(0))- \eta^{a}(\xi(0))$;
    	\item $1/M \leq \eta^{a}(\xi(0)) $;
    	\item $\eta^{a+1}(\xi(0)) \leq - 1/M$;
    	\item $\| \xi\|_{C^3([\eta(0),0])} \leq M$;    	
    	\item $\| \eta\|_{C^3([0, \xi(0)]\setminus \{\beta\})} \leq M$; 
    	\item $1/M \leq D \eta(x)$, for all $x \in [\eta^{a}(\xi(0)),\xi(0)] \setminus \{\beta\}$. %_1, \dots, \beta_{N-1}\}$.
    \end{enumerate}
If $\beta \notin (0, \eta^{a+1}(\xi(0)))$ or $\beta \notin (\eta(\xi(0)), \xi(0))$, % for some $i \in \{1, \cdots, N-1\}$, 
then the following extra condition is also satisfied:
\begin{enumerate}
    	\item [(9)] $\eta(\beta)- \beta \geq 1/M$ and $\beta - \eta^{-1}(\beta) \geq 1/M$.    	
    \end{enumerate}
\end{definition}

Analogously, in the case that $\beta \in (\eta(0),0)$, we can extend  the definition of $M$-controlled bi-critical commuting pair.

 The next result guarantees that every $C^3$ bi-critical circle map with irrational rotation number gives rise to a sequence of bi-critical controlled commuting pairs, after a finite number of renormalizations.

    \begin{theorem}\label{thm:Mcontrolledmcm}
    	There exists a universal constant $M > 1$ with the following property. For any given $C^3$ bi-critical circle map $f$ with irrational rotation number there exists $n_0=n_0(f) \in \nt$ such that the critical commuting pair $\mathcal{R}^nf$ is $M$-controlled for any $n \geq n_0$.
    \end{theorem}
    
    \begin{proof}[Proof of Theorem \ref{thm:Mcontrolledmcm}]
    	By universal real bounds (Theorem \ref{teobeau}), there exists a constant $C>1$ and $n_0 \in \nt$ such that for any $n \geq n_0$ properties $(1)$ to $(5)$ and property $(8)$ hold. Moreover, by \cite[Lemma 4.2, Lemma 4.1]{EdF} and  real bounds we know that $\left|[\eta(\beta), \beta] \right|$ is comparable with  $\xi(0)$, then by Property $(1)$ $|\eta(\beta)- \beta|\geq 1/M$. By \cite[Lemma 3.3]{EdF}, $\left|\beta - \eta^{-1}(\beta)\right| \geq 1/M$, and hence property $(9)$ is also true. Finally, from \cite[Lemma 5.1]{EdFG} and \cite[Theorem A.4 (pages 389-381) and remark (page 381)]{dFdM}, there exists $n_0 \in \nt$ such that for all $n \geq n_0$ properties $(6)$ and $(7)$ hold. 
    \end{proof}
   
   \begin{remark}
   	 By \cite[Theorem A.4 (pages 389-381) and remark (page 381)]{dFdM}, the renormalizations of a $C^3$ uni-critical circle map are controlled as in our definition. That result and its proof can be generalized to the bi-critical case.
   \end{remark}
  
     As $\mathcal{R}^nf$ and $\mathcal{R}^ng$ satisfy Theorem \ref{thm:Mcontrolledmcm} and Inequality \eqref{ineq:convergencerenormalizations} for $d=2$, we know that $\lim_{n \to \infty} \log \frac{|h(I_{n}(c_i))|}{|I_{n}(c_i)|}$ exists and converge exponentially. Conjugating the map $f$ (or $g$) by a $C^{\infty}$ diffeomorphism we can achieve $\lim_{n \to \infty} \log \frac{|h(I_{n}(c_i))|}{|I_{n}(c_i)|}=0$ at an exponential rate. Therefore, we can always assume that there exist $C>0$ and $\lambda \in \{0,1\}$ such that for any $n \geq n_0$ and $i \in (0,1)$ we have
   \begin{equation} \label{eq:limitfundintervals}
   	\left| \log \frac{|h(I_{n}(c_i))|}{|I_{n}(c_i)|} \right| \leq C\, \lambda^n.
   \end{equation}

%%%%%%%%%%%%%%%%%%%%%%%%%%%%%%%%%%%%%%%%%%

\section{The two-bridges partition $\widehat{\mathcal{P}}_n$}\label{sec:nuevaparticion}
	In this section we recall the construct of the sequence of the two-bridges partitions $\{\widehat{\mathcal{P}}_n\}_{n \in \nt}$, defined in \cite[Section 4.1.1]{EG}. Let $c_0$ and $c_1$ be the two critical points of $f$. For $n \in \nt$, the first return map to $J_n(c_0)$, which is given by 
	$(f^{q_{n+1}}|_{I_n(c_0)}, f^{q_n}|_{I_{n+1}(c_0)})$, has two critical points: one being $c_0$ and the other one being the unique pre-image of $c_1$ in $J_n(c_0)$. Such a preimage  in  $J_n(c_0)$ is called \textit{the free critical point at level $n$}, and is denoted by $\mathfrak{c}_{n,0}$. Without lost of generality, let us assume that $\mathfrak{c}_{n,0}$ belongs to $I_n(c_0)$. This means that there exists $j < q_{n+1}$ such that $c_1 \in f^{j}(I_n(c_0))$ and $f^{j}(\mathfrak{c}_{n,0})=c_1$. Note that if $c_1$ belongs to the positive orbit of $c_0$ then there exists $m \in \mathbb{N}$ such that for all $n \geq m$  the return maps $(f^{q_{n+1}}|_{I_n(c_0)}, f^{q_n}|_{I_{n+1}(c_0)})$ are unicritical and we will be in the situation studied in \cite{KT}. So, from now on, let us assume that $c_1$ does not belong to the positive orbit of $c_0$. Moreover we will assume that $n \geq n_0$ where $n_0$ is given by Theorem \ref{thm:Mcontrolledmcm}.

	\begin{definition}[Two-bridges level] A natural number $n$ is a two-bridges level for $f$ at $c_0$ if $a_{n+1} \geq 23$, the free critical point $\mathfrak{c}_{n}$ belongs to $I_n(c_0)\setminus I_{n+2}(c_0)$ and moreover $\mathfrak{c}_{n} \in \{\Delta_{11}, \Delta_{12}, \dots, \Delta_{a_{n+1}-10} \}$, where $\Delta_j = f^{(j-1)q_{n+1}+q_n}(I_{n+1}(c_0))$ for all $j \in \{1, \dots, a_{n+1}\}$.
	\end{definition}
	
	% We define $\widehat{\mathcal{P}}_0$ as being the standard partition $\mathcal{P}_0(c_0)$, that is,
	%\[
%	\widehat{\mathcal{P}}_0= \mathcal{P}_0(c_0)= 
	%\{[f^i(c_0), f^{i+1}(c_0)] : i \in \{0, \dots, a_0 - 1\} \} \cup \{ [f^{a_0}(c_0), c_0] \}.
%	\]

	We recall the inductive construction of the two-bridges partition $\widehat{\mathcal{P}}_{n}$ for $n \geq 0$. To start we define $\widehat{\mathcal{P}}_{0}$ be the same as $\mathcal{P}_{0}$, now for $n \in \nt$ we have two cases:
	
	\noindent \textbf{Case 1:} $n$ is a two-bridges level for $f$ at $c_0$. Then we consider the following six pairwise disjoint intervals in $I_n(c_0)\setminus I_{n+2}(c_0)$:
	\[
	\Delta_1=f^{q_n}(I_{n+1}(c_0)),
	\ \ \widehat{\Delta}_n=[f^{q_{n+1}}(\mathfrak{c}_{n}),\mathfrak{c}_{n}]
	\]
	\[
	\Delta_{a_{n+1}} = f^{(a_{n+1}-1)q_{n+1}} (\Delta_1) = [f^{q_{n+2}}(c_0), f^{(a_{n+1}-1)q_{n+1}+q_n}(c_0)].
	\]
	For each $j \in \Z$, denote by $\Delta_1^j$, $\widehat{\Delta}_n^j$ and $\Delta_{a_{n+1}}^j$  the intervals $f^{jq_{n+1}}(\Delta_1)$, $f^{jq_{n+1}}(\widehat{\Delta}_n)$ and $f^{jq_{n+1}}(\Delta^{a_{n+1}})$, respectively.
	Let $r(n), \ell(n) \in \{0, \dots, a_{n+1}\}$ be given by
	\[
	r(n) = \min\{j \in \nt : \widehat{\Delta}_{n}^{-j} \cap \Delta_1^j \neq \emptyset \} \text{ \ and \  } \ell(n) = \min\{j \in \nt : \Delta_{a_{n+1}}^{-j} \cap \widehat{\Delta}_n^{j} \neq \emptyset \}.
	\]
	Note that the intersections above may be given by a single point (in this case the critical points have the same orbit). Just to fix ideas, let us assume that $\Delta_{a_{n+1}}^{-\ell(n)} \setminus \widehat{\Delta}_n^{\ell(n)} \subset \widehat{\Delta}_n^{\ell(n) + 1}$ \ and \ $\widehat{\Delta}_{n}^{-r(n)} \setminus \Delta_1^{r(n)}  \subset \Delta_{1}^{r(n)+1}$, we consider $\Delta_{R}^{n+1} =
	\Delta_1^{r(n)} \cup \widehat{\Delta}_n^{-r(n)}$ and $\Delta_{L}^{n+1}= \Delta_{a_{n+1}}^{-\ell(n)} \cup \widehat{\Delta}_n^{\ell(n)}$. We define 		$\widehat{\mathcal{P}}_{n+1}$ inside $I_n(c_0)$, for a two-bridges level $n$, as 
	\begin{equation*}
	\widehat{\mathcal{P}}_{n+1}|_{I_n(c_0)} =
	\{I_{n+2}(c_0), \{ \Delta_{a_{n+1}}^{-j} \}_{j=0}^{j=\ell(n) - 1}, \Delta_L^{n+1},
	\{ \widehat{\Delta}_n^j \}_{j=1-r(n)}^{j=\ell(n)-1}, \Delta_R^{n+1},  \{ \Delta_{1}^{j} \}_{j=0}^{j=r(n)-1} \}.
	\end{equation*}
%	\begin{figure}[!ht]
%		\centering
%		\psfrag{In+2}[][][1]{$I_{n+2}$} 
%		\psfrag{Deltaan+1}[][][1]{$\Delta_{a_{n+1}}$}
%		\psfrag{R}[][][1]{$\Delta_{a_{n+1}}^{-\ell+1}$}
%		\psfrag{Deltal}[][]{$\Delta_L$} 
%		\psfrag{S}[][][1]{$\small{\widehat{\Delta}_{0}^{\ell-1}}$}
%		\psfrag{F}[][][1]{$f^{q_{n+1}}$}
%		\psfrag{G}[][][1]{$f^{-q_{n+1}}$} 
%		\psfrag{Delta0}[][][1]{$\Delta_1$}
%		\psfrag{c}[][][1]{$\mathfrak{c}_{n,0}$}
%		\psfrag{Delta0}[][][1]{$\Delta_1$}
%		\psfrag{Deltahat}[][][1]{$\widehat{\Delta}_1$}
%		\psfrag{Deltahat-1}[][][1]{$\widehat{\Delta}^{-1}_1$}
%		\psfrag{V}[][][1]{$\widehat{\Delta}_1^{-1}$}
%		\psfrag{T}[][][1]{$\small{\widehat{\Delta}_{1}^{-r+1}}$}
%		\psfrag{Deltar}[][]{$\Delta_R$}
%		\psfrag{U}[][]{$\Delta_1^{r-1}$} 
%		\psfrag{...}[][]{$\dots$} 
%		\includegraphics[width=5.1in]{decomposition.eps}
%	\end{figure}
	
	\begin{figure}[!ht]
	  \includegraphics[width=6.2in]{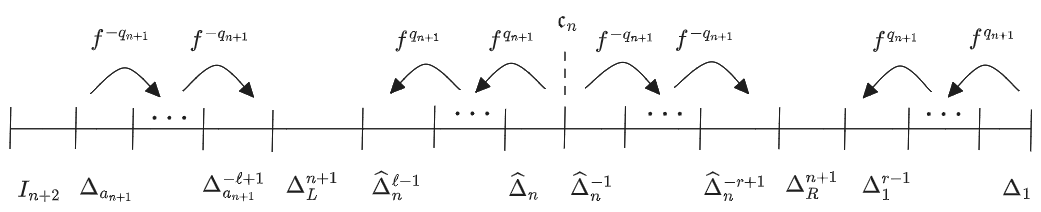}
	  \caption{\label{decomposition}  The two-bridges partition $\widehat{\mathcal{P}}_{n+1}|_{I_n(c_0)}$ in Case 1.}
	\end{figure}
	
	Now, we spread the definition of $\widehat{\mathcal{P}}_{n+1}$ to the circle as usual, that is:
	\[
	\widehat{\mathcal{P}}_{n+1}= \{f^i(I_{n+1}(c_0)) : 0 \leq i \leq q_{n} -1\} \ \bigcup \ \{f^j(I)|_{I \in \widehat{\mathcal{P}}_{n+1} \cap I_n(c_0)}: 0 \leq j \leq q_{n+1}-1\}.
	\]
	\noindent \textbf{Case 2:}  $n$ is not a two-bridges level for $f$ at $c_0$, then  $\widehat{\mathcal{P}}_{n+1}$ coincides with $\mathcal{P}_{n+1}$ except in the following cases. For any given endpoint $v$ of $\widehat{\mathcal{P}}_{n}$, let $I' \in \mathcal{P}_{n+1}$ be the interval containing $v$ and let $w$ be the endpoint of $I'$ closest to $v$ (in the Euclidean distance). 
	\begin{itemize}
		\item If $v$ is the mid point of $I'$, then we add $v$ to the set of endpoints of $\widehat{\mathcal{P}}_{n+1}$.
		\item If $v$ is not the middle point of $I'$: we remove $w$ from the set of endpoints of $\widehat{\mathcal{P}}_{n+1}$ and we add $v$ to the set of endpoints of the partition $\widehat{\mathcal{P}}_{n+1}$.
	\end{itemize}

\subsection{Properties of the two-bridges partition}

    For any $n \in \nt$ we denote by   $\widehat{E}_{n}$ the set of the endpoints of $\widehat{\mathcal{P}}_{n}$. We observe that, from the construction, when we lost one endpoint of $E_{k}$ by defining $\widehat{E}_{k}$, that missing endpoint will appear in the next level, that is in $\widehat{E}_{k+1}$. 
    From the construction we have the following properties of $\{\widehat{\mathcal{P}}_n\}_{n \in \nt}$:
	\begin{enumerate}
		\item Each partition $\widehat{\mathcal{P}}_n$ is dynamically defined: each $v\in \widehat{E}_n$  is an iterate (either forward or backward) of a first return map at $c_0$ or $c_1$.
		\item The intervals $I_n(c_0)$ and $I_{n+1}(c_0)$ belong to $\widehat{\mathcal{P}}_{n+1}$.
		\item The sequence of partitions $\{\widehat{\mathcal{P}}_{n}\}$ is nested.
		\item The sequence of partitions $\{\widehat{\mathcal{P}}_{n}\}$ is refining.
		\item If $n$ is a two-bridges level for $f$ at $c_0$ then  $\mathfrak{c}_{n} \in \widehat{E}_{n+1}$.
		\item Any endpoint of the standard partition $\mathcal{P}_n$ belongs to $\widehat{E}_m$ for some $m \geq n$.
		\item (Real bounds) There exists an universal constant $\widehat{C} > 1$, such that $\widehat{C}^{-1} |I| \leq|J | \leq \widehat{C} |I|$ for each pair of adjacent atoms $I, J\in \widehat{\mathcal{P}}_n$.
		%\item Let $J \in \widehat{\mathcal{P}}_{n}$ and $I \in \widehat{\mathcal{P}}_{n+1}$ with $I \subseteq J$.  There exists $\mu=\mu(\widehat{C}) \in (0,1)$ such that $|I| \leq \mu \ |J|$.
	\end{enumerate}

	\begin{corollary}\label{coro_twolevelwidehatP} (Exponential refining) There exist a constant $\widehat{\mu} \in (0,1)$ and $n_0 \in \nt$, both universal in $\mathcal{F}_{2,d}$, such that for all $m>n\geq n_0$: if $\widehat{\mathcal{P}}_{m} \ni I \subseteq J \in \widehat{\mathcal{P}}_{n}$, then $|I| \leq \widehat{\mu}^{ m-n}|J|$. 
\end{corollary}

For different levels of the two-bridges partition, we have

\begin{lemma}\cite[Lemma 4.6]{EG} \label{lema:difniveles}
	 Let $n, p \in \nt$, and $v$ be an endpoint of $\widehat{\mathcal{P}}_{n+p}$ contained in $J_n(c_0)$. There exist $L \in  \{1,..., p\}$ and $n \leq m_1 \leq \dots \leq m_L$, such that $v = \phi_1 \circ \dots \circ \phi_L (x)$, 	where
	 \begin{enumerate}
	 	\item For each $j \in \{1, \dots, L\}$ we have $\phi_j = f^{k_j q_{m_j +1} + \sigma_j q_{m_j}}$ for $\sigma_j \in \{0, 1\}$ and $k_j \in \Z$ either satisfying $|k_j| \leq \ell(m_j)$ and $|k_j| \leq r(m_j)$ or $|k_j| \leq \lceil a_{m_{j}+1}/2 \rceil$, depending on whether $m_j$ is or is not a two-bridges level for f at $c_0$.
	 	\item For each $j \in \{1, \dots, L -1\}$, the point $\phi_{j+1} \circ \dots \circ \phi_L (x)$ either belongs to $J_{m_j+1}(c_0) \cup \widehat{\Delta}_{m_j}$ or to $J_{m_j+1}(c_0)$, depending on whether $m_j$ is or is not a two-bridges level.
	 	\item There exists $m \in \{m_L, \dots, n + p\}$, such that $x$ belongs to $\{c_0, f^{q_{m+2}}(c_0), \mathfrak{c}_{m} \}$ or to $\{c_0, f^{q_{m+2}} (c_0)\}$, depending on whether $m$ is or is not a two-bridges level.
	 \end{enumerate}
\end{lemma}

\begin{remark}\label{rmk:comparabilitytwopartions} The relation between the length of interval is the classical dynamical partition and the two-bridges partition is the following: let $C>1$ and $\widehat{C}>1$ be the two constants given by real bounds of the classical and the two-bridges partition, respectively. If $I \in \mathcal{P}_n$ and $\widehat{I} \in \widehat{\mathcal{P}}_n$ and $I \cap \widehat{I}  \neq \emptyset$ then
	\[
	\frac{1}{\max\{C+1, \widehat{C}+1\}} \leq \frac{|I|}{|\widehat{I}|} \leq \max\{C+1, \widehat{C}+1\}.
	\]
\end{remark}	

%%%%%%%%%%%%%%%%%%%%%%%%%%%%%%%%%%%%%%%%%%%%%%%
\section{Idea of the proof}\label{sec:idea}

The proof of our Main Theorem consists in to construct a sequence of circle partitions that satisfy some specific properties. More precisely, we want to prove the following result:

\begin{proposition}\label{criterion} \cite[Proposition, Page 199]{KT} Let $h$ be a circle map and $\{\mathcal{Q}_n^f\}_{n \in \nt}$ and $\{\mathcal{Q}_n^g\}_{n \in \nt}$ be two sequences of circle partitions such that for all $I \in \mathcal{Q}_n^f$ we have $h(I) \in \mathcal{Q}_n^g$. Suppose also that:
\begin{enumerate}
	\item  $\{\mathcal{Q}_n^f\}_{n \in \nt}$ and $\{\mathcal{Q}_n^g\}_{n \in \nt}$ are sequence of refinement circle partitions;
	\item  the maximal length of an interval in $\mathcal{Q}_n^f$, and the maximal  length of an interval in $\mathcal{Q}_n^g$,  go to zero as $n$ goes to infinite;
	\item  there exist  $B> 0$ and $\mu \in (0, 1)$ such that for any two intervals $I, I' \in \mathcal{Q}_n^f$, which either are adjacent or $I, I' \subset J$ for some $J \in \mathcal{Q}_{n-1}^f$,  we have
	\begin{equation*} 
		\left|\log \frac{|h(I)|}{|I|} - \log \frac{|h(I')|}{|I'|}\right| \leq B\, \mu^{n};
	\end{equation*}
\end{enumerate}	
	then $h$ is a $C^{1}$ circle diffeomorphism.
\end{proposition} 
%\textcolor{red}{Let us just recall that (as \cite[Remark 4]{KT}) Proposition \ref{criterion} is true even if adjacent intervals of the circle partition are not comparable.}

We will prove that  the sequence of two-bridges partition $\{\mathcal{\widehat{P}}_n^f\}_{n \in \nt}$, and its corresponding sequence  $\{\mathcal{\widehat{P}}_n^g\}_{n \in \nt}$\footnote{If the conjugacy $h$ identifies critical points of $f$ with critical points of $g$, then $h$ also identifies the corresponding two-bridges partitions  $\mathcal{\widehat{P}}_n^f$ and $\mathcal{\widehat{P}}_n^g$, for all $n \in \nt$.} for $n \geq n_0$ where $n_0$ is given by Theorem \ref{thm:Mcontrolledmcm}, satisfy the three conditions of Proposition  \ref{criterion}. The only condition that is not a immediate consequence of the definition and properties of $\{\mathcal{\widehat{P}}_n^f\}_{n \in \nt}$ is the third one, the main technical result in this paper is that condition $(3)$ of Proposition \ref{criterion} holds.

%As it was observed in \cite[Section 3.1.2]{KT} %and in \cite[Lemma 4.7]{dFdM}, respectively, the exponential convergence of renormalizations implies that
% there exists $S>0$ such that:
%\begin{equation*} 
%	\left|\log \frac{|h(I_n)|}{|I_n|} - \log %\frac{|h(I_{n+1})|}{|I_{n+1}|}\right| = | \log \mathcal{R}^n(f) - \log \mathcal{R}^{n}(g)| \leq S\, \lambda^{n}.
%\end{equation*}
%Therefore, $s = \lim_{n \to \infty} \log \frac{|h(I_n)|}{|I_n|}$  exists, and the convergence is exponential. Conjugating $f$ or $g$ by a diffeomorphism we can assume $s = 0$. This change of coordinates  does not interfer with the renormalizations, e.g. \cite[Lemma 3.1]{EG}. So, from now on, we will prove that 
%there exist $n_0 \in \nt$,  $C > 0$ and $\lambda \in (0,1)$ such that for all $n \geq n_0$,
%\begin{equation} \label{eq:carlesonpaso1}
%	\left|\log \frac{|h(I_n(c_0))|}{|I_n(c_0)|} \right| \leq C\, \lambda^{n}.
	% \, \text{ \ and  \ } \,
%	\left| \frac{|h(I_{n}(c_0))|}{|I_{n}(c_0)|}- 1 \right| \leq C\, \lambda^{n}.
%\end{equation}
%Since Inequality \eqref{eq:carlesonpaso1} implies  condition  $(3)$, in Proposition \ref{criterion},  for $I=I_{n+2}(c_0)$ and $I'=I_{n+1}(c_0)$, then  we just need to prove Inequality \eqref{eq:carlesonpaso1} for $I, I' \in \mathcal{\widehat{P}}_{n+1}^f \setminus\{I_{n+2}(c_0), I_{n+1}(c_0)\}$.

 Let us add some words about how we will obtain condition  $(3)$:
\begin{enumerate}
	\item First, we fixed $m \in \nt$, and we the consider the  two fundamental domain of the renormalizations $J_m(c_0)$ and $J_m(c_1)$. Inspired by \cite{KT}, in  Section \ref{subsec:unbounded} we prove that there exists an uniform strip around $m$ where, for all $n$ in that strip, the distance between each  endpoint in $\widehat{E}_{n}$ (contained in $J_m(c_0)$ or in $J_m(c_1)$) and its image by $h$ is exponentially small, with rate $n$.  We get those estimates in Proposition \ref{pro_pasouno}.
%	\item Second, we get an estimate of the differences  between the $m-$renormalizations of each  endpoint in $\widehat{E}_{n}$ and its image by $h$, for $n$ in an uniform strip around $m$. The estimate  goes to zero at exponential rate in $n$ (Proposition \ref{pro_pasouno}).
	\item Second, using the step before and the shape of the renormalizations inside $J_m(c_0)$ and $J_m(c_1)$, we get inequality $(3)$ in  Proposition  \ref{criterion} %\eqref{eq:carlesonpaso1} 
	for intervals in $\mathcal{\widehat{P}}_{n}^f$ which are contained in $J_m(c_0) \cup J_m(c_1)$. Those estimates are proved in Proposition \ref{Pro_pasocuatro}.
	\item Finally, using Koebe Distortion Principle (Lemma \ref{KoebeDistortionPrinciple}) and the fact that we can take any pair of intervals of   $\mathcal{\widehat{P}}_{n}^f$  inside $J_m(c_0)$ or $J_m(c_1)$, by a diffeomorphism, we extend the estimates obtained in second step  to intervals outside $J_m(c_0) \cup J_m(c_1)$.
\end{enumerate}
To obtain the proof of step 1 (Section \ref{sec:key}), we need some background given in Section \ref{subsec:unbounded}. The proof of step 2 is given in Section \ref{sec:estimatesintervals}. The final step will be proved in Section \ref{sec:proofmaintheorem}. 

%%%%%%%%%%%%%%%%%%%%%%%%%%%%%%%%%%%%%%%%%%%%%%%%%
 
 \section{Un-bounded rotation number} \label{subsec:unbounded}
 
 From now on, for all $n \in \nt$, we will renormalize around $c_0$, so we write $I_n^f$, $I_{n+1}^f$, $J_n^f$, $\mathfrak{c}_{n}$, $s_{n}^f$ and $\mathcal{R}^{n}f$ instead of $I_n^f(c_0)$, $I_{n+1}^f(c_0)$, $J_n^f(c_0)$, $\mathfrak{c}_{n,0}^{f}$, $s_{n,0}^f$ and $\mathcal{R}_0^{n}f$ respectively. Since $h$ is the conjugacy that identifies corresponding critical points,  then  $\mathfrak{c}_{n}^g=h(\mathfrak{c}_{n})$. We notice that the construction, definitions and results given in this section also apply if we renormalize around  $c_1$. 
 
  We say that an irrational number  $\rho=[a_0,a_1, \dots, a_n, \dots] \in [0,1)$ is \textit{bounded type} if there exists  $A \in \nt$ such that $\sup_{k} a_{k} \leq A$. The set of bounded type rotation numbers has zero Lebesgue measure in $[0,1]$. As we mentioned in the introduction, we focus our attention on un-bounded type rotation numbers, i.e. $\rho=[a_0,a_1, \cdots, a_n, \cdots]$ such that there is no $A \in \nt$ such that $\sup_{k}a_{k}< A$. In particular, the fundamental interval $I_{k}$ is subdivided by an un-bounded number of sub-intervals of the classical dynamical partition $\mathcal{P}_{k+1}$, and hence the graph of $\mathcal{R}^{k}f$ is arbitrarily close to the identity.
   	
  For a given $n \in \nt$ and $L \in \nt$, we define the $L-$th \textit{tubular} of $\mathcal{R}^nf$ as the set:$$M_{n,L}^f= \{ z \in [0,1]: z-\mathcal{R}^nf(z)< 1/L \}.$$
    We observe that if $a_{n+1}>L$ then the set $M_{n,L}^f$ is not empty. %If $L$ is small enough then the funnel can have at most $1$ connected component.

  \begin{lemma} \label{lemma:tubularcoordinates}
  	Given $M>1$ such that $\mathcal{R}^nf$ is $M$-controlled, there exists $L=L(M) \in \nt$ such that:
  	\begin{enumerate}
  		\item  $r(n)\geq L$, and in this case there exists an unique point $p \in Int [(\mathcal{R}^nf)^{-1}(B_{n,f}(\mathfrak{c}_n)),1]$ such that $|\mathcal{R}^nf(p)-p| \leq |\mathcal{R}^nf(z)-z|$ for all $z \in [(\mathcal{R}^nf)^{-1}(B_{n,f}(\mathfrak{c}_n)),1]$, $D \mathcal{R}^nf(p)=1$ and $D^2\mathcal{R}^nf(p)<0$.
  		\item If $\ell(n)\geq L$ then there exists an unique point $q \in Int [0,\mathcal{R}^nf(B_{n,f}(\mathfrak{c}_n))]$ such that $|\mathcal{R}^nf(q)-q| \leq |\mathcal{R}^nf(z)-z|$ for all $z \in [0,\mathcal{R}^nf(B_{n,f}(\mathfrak{c}_n))]$, $D \mathcal{R}^nf(q)=1$ and $D^2\mathcal{R}^nf(q)<0$.
  	\end{enumerate}
  \end{lemma}

\begin{proof}
	Note that $S (\mathcal{R}^nf)<0$ in $[0,\mathcal{R}^nf(B_{n,f}(\mathfrak{c}_n))] \cup [(\mathcal{R}^nf)^{-1}(B_{n,f}(\mathfrak{c}_n)),1]$. The proof then follows from a small adaptation of the proof of  \cite[Lemma 6.1]{GMdM}.
\end{proof}

By Lemma \ref{lemma:tubularcoordinates} and \cite[Remark 6.2]{GMdM}

\begin{corollary}\label{coro:centro}
	Given $M>1$ there exists $L=L(M) \in \nt$ such that $M_{n,L}^f$ is a non empty set, which can be an open interval or two open intervals, that contains at most two points $z_{n}^f$ and $w_n^f$ which attains the minimum of $|z-\mathcal{R}^nf(z)|$ inside $M_{n,L}^f$. For those points we have  $D\mathcal{R}^nf(z_{n}^f)=1$, $D\mathcal{R}^nf(w_{n}^f)=1$, $D^2\mathcal{R}^nf(z_{n}^f)<-1/M$ and $D^2\mathcal{R}^nf(w_{n}^f)<-1/M$.
\end{corollary}

 If the points $z_{n}^f$ and $w_n^f$ in Corollary \ref{coro:centro} exist, then they are called the \textit{centers} of $M_{n,L}^f$.  The exponential convergence of the renormalizations, in the $C^2$-topology, implies that 
 \begin{equation}\label{inq:estimativacentrofunnel}
 	\left| z_{n}^g - z_{n}^f \right| = \left| h(z_{n}^f) - z_{n}^f \right| \leq C\, \lambda^n
 \end{equation}
and
\begin{equation}\label{inq:estimativacentrofunnel2}
	\left| w_{n}^g - w_{n}^f \right| = \left| h(w_{n}^f) - w_{n}^f \right| \leq C\, \lambda^n,
\end{equation}
where $h(z_{n}^f)$ and $h(w_{n}^f)$ are the centers of $M_{n,L}^g$. From now on we denote the centers $z_{n}^f$ and $w_n^f$  by $z_{n}$ and $w_n$, respectively. 
 
% \begin{figure}[!ht]
 %	\includegraphics[width=6.1in]{centrotubularrn.PNG}
 %	\caption{\label{decomposition}  In this figure $r(n)>L$ and then the center $z_{n}$ of  $M_{n,L}^f$ is defined}
 %\end{figure}
 
 \begin{figure}[!ht]
 	\includegraphics[width=3.8in]{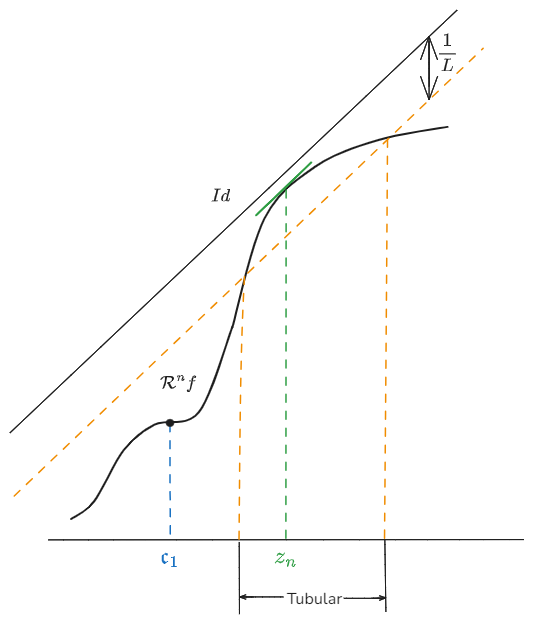}
 	\caption{\label{decomposition}  In this figure $r(n)>L$ and then the center $z_{n}$ of  $M_{n,L}^f$ is defined}
 \end{figure}

 \subsection{Tubular coordinates} 
 Let us fix $m \in \nt$ and $L=L(a_{m+1}) \in (0,1)$ such that $M_{m,L}^f$ is not empty and the center $z_{m}$ is defined. A similar construction can be done if $w_{m}$ is defined.
  
 \begin{definition} \label{def:tubularcoordinates}
  The tubular-coordinates of $\mathcal{R}^mf$ with respect to $z_{m}$, is defined as the change of coordinates of $\mathcal{R}^mf$ given by:  
 \begin{equation*}
 	 \mathcal{F}_m(x) = \phi_m \circ \mathcal{R}^mf \circ \phi_m^{-1}(x),
 \end{equation*}
where $\phi_m(x)= (D^2\mathcal{R}^mf)(z_{m})(x- z_{m})/2$, \ for  $x\in [ B_{m,f}(f^{-q_{m+1}}(\mathfrak{c}_m)), B_{m,f}(f^{q_m}(c_0)) ]$. %if $r(m)>2L$ and for $x\in [B_{m,f}(f^{-Lq_{m+1}+q_{m+2}}(c_0)), B_{m,f}(f^{Lq_{m+1}}(\mathfrak{c}_m))]$ if $\ell(m)>2L$.
\end{definition}

Analogously, we define the  tubular-coordinates of $\mathcal{R}^mg$ with respect to $h(z_{m})$, as the change of coordinates given by
 \begin{equation*}
	\mathcal{G}_m(x) = \psi_m \circ \mathcal{R}^mg \circ \psi_m^{-1}(x),
\end{equation*}
where $\psi_m(x)= (D^2\mathcal{R}^mg)(h(z_{m}))(x- h(z_{m}))/2$, \ for  $x\in [ B_{m,g}(g^{-q_{m+1}}(h(\mathfrak{c}_m))), B_{m,g}(g^{q_m}(h(c_0))) ]$.

Since  a change of coordinates does not affect the original exponential convergence (see e.g. \cite[Lemma 3.1]{EG}), then the exponential convergence  of $\mathcal{R}^mf$ and $\mathcal{R}^mg$ in the $C^1$-topology, implies that the tubular-coordinates of $\mathcal{F}_m$ and  $\mathcal{G}_m$ also converge exponentially in the same topology.

We are interested in obtaining asymptotic estimates of $\mathcal{F}_m$ and of $\mathcal{G}_m$, near the origin.
Since $\mathcal{F}_m'(0)=1$, $\mathcal{F}_m''(0)=2$ and  $\varepsilon_m = \mathcal{F}_m(0) = \min_x \{\mathcal{F}_m(x)-x\}<0$, then there exists  $\mathfrak{C}>0$ such that
 \begin{equation}\label{eq:Taylor} 
 	|\mathcal{F}_m(x) - (\varepsilon_m + x + x^2)| \leq \mathfrak{C} |x|^{2+\alpha} \text{, \ for $x \in \phi_m([0,1])$ }.
 \end{equation}

 Note that if $|x|^{2+\alpha}>C_0 \, |\varepsilon_m |$, for some $C_0>0$, then from \eqref{eq:Taylor} we get $|\mathcal{F}_m(x) - (x + x^2)| \leq \widetilde{C}_0 |x|^{2+\alpha}$, so we can dispose of the value  $|\varepsilon_m|$. For those points we use the following estimates (see \cite[Lemma 5]{KT}).
 
 \begin{lemma}[Estimates inside the Funnel]\label{lemaKTfunnel}
 	Let $\{s_i\}_{i \geq 0}$ be a sequence of real numbers such that there exist $\widehat{C_0}>0$  with $|s_{i+1} - (s_i - s_i^2 )| \leq \widehat{C_0} \,|s_i|^{2+\alpha}$ for every $i \geq 0$. Then there exist positive constants $D_1$, $D_2$ and $d_1 \in (0, 1)$ (all depending on $\widehat{C_0}~$ and $\alpha$) such that as long as $s_0 \in (0, d_1]$, then for all $i \geq 0$ we have 
 	\begin{equation}\label{eq_funnelsis0}
 		\left|s_i - \frac{1}{i+s_0^{-1}} \right| \leq  \frac{D_1}{(i+s_0^{-1})^{1+\alpha}} 
 	\end{equation}	
 	and
 	\begin{equation}\label{eq_funnelsonsecutive}
 		s_i - s_{i+1} =  \frac{1+\delta_i}{(i+s_0^{-1})^{2}}, \text{ \ for $|\delta_i| \leq D_2 s_0^{ \alpha}$}.
 	\end{equation}
 \end{lemma}
 
 On the other hand, if  $|x|^{2+\alpha} <C_0 \, \varepsilon_m$, for some $C_0>0$, then  for those points we can use the following estimates (see \cite[Lemma 6]{KT}).
 
  \begin{lemma}[Estimates inside the Tunnel]\label{lemaKTtunnel}
 	Let $\{s_i\}_{i \geq 0}$ be a sequence of real numbers such that there exist $\widetilde{C_1},\widetilde{C_2}> 0$ and $\varepsilon \in (0,1)$ with
 	\begin{enumerate}
 		\item  $|s_0| \leq \widetilde{C_1} \, \varepsilon$,
 		\item $|s_{i+1} - (\varepsilon + s_i + s_i^2 )| \leq \widetilde{C_2}|s_i|^{2+\alpha}$, for all $i \geq 0$.
 	\end{enumerate}
 	Fix arbitrary $\widetilde{C_3} > 0$ and define $N = N(\widetilde{C_3},\varepsilon) = \varepsilon^{-1/2} \, tan^{-1}(\widetilde{C_3} \varepsilon^{-\alpha/2(2+\alpha)})$. Then there exist positive constants $D_3, D_4$ and $d_2 \in (0,1)$ (all depending on $\widetilde{C_1}, \widetilde{C_2}, \widetilde{C_3}$ and $\alpha$) such that, as long as $\varepsilon \in (0, d_2]$, for every $0 \leq i \leq N$:
 	\begin{equation}\label{eq_tunnelsis0}
 		\left|s_i - \sqrt{\varepsilon} \tan\left(\sqrt{\varepsilon}i +
 		\tan^{-1}\left( \frac{s_0}{\sqrt{\varepsilon}}\right) \right) \right|  \leq D_3(\sqrt{\varepsilon} \tan \sqrt{\varepsilon}i)^{1+ \frac{\alpha(\alpha +1)}{2}}
 	\end{equation}
 	and
 	\begin{equation}\label{eq_tunnelconsecutive}
 		s_{i+1} - s_i = \frac{\varepsilon (1 + \delta_i)}{(\cos \sqrt{\varepsilon} i)^2} , \text{ \  \ \ where $|\delta_i| \leq D_4\varepsilon^{\frac{\alpha(\alpha+1)}{2(2+\alpha)}}$ .}
 	\end{equation}
 \end{lemma}
 
 The proofs of Lemma \ref{lemaKTfunnel} and Lemma \ref{lemaKTtunnel} can be found in \cite[Appendix]{KT}.
 
 Let us give some more known estimates. First, let $i_c, \widehat{i_c} \in \nt$  the smallest natural number such that $\mathcal{F}_m^{i_c}(\phi_m(1)) \in (0,-\varepsilon_m]$ and $\mathcal{F}_m^{-\widehat{i_c}}(\phi_m(B_{m,f}(\mathfrak{c}_m))) \in (0,-\varepsilon_m]$, respectively. Moreover,  let $i_{r}, i_{\ell} \in \nt$ be such that  $\mathcal{F}_m^{i_{r}}(\phi_m(1))$ is the most right endpoint of the Funnel and $\mathcal{F}_m^{i_{\ell}}(\phi_m(1))$ is the most left endpoint of the Funnel. Similarly, let  $\widehat{i_{r}}, \widehat{i_{\ell}} \in \nt$ such that  $\mathcal{F}_m^{-\widehat{i_{\ell}}}(\phi_m(B_{m,f}(\mathfrak{c}_m)))$ is the most left endpoint of the Funnel and $\mathcal{F}_m^{-\widehat{i_{r}}}(\phi_m(B_{m,f}(\mathfrak{c}_m)))$ is the most right endpoint of the Funnel.\\
 Analogously, let us define $\sigma_m = \mathcal{G}_m(0) = \min_x \{\mathcal{G}_m(x)-x\}<0$. Let $j_c, \widehat{j_c} \in \nt$ be the smallest natural numbers such that $\mathcal{G}_m^{j_c}(\psi_m(1)) \in (0,-\sigma_m]$ and $\mathcal{G}_m^{- \widehat{j_c}}(\psi_m(B_{m,f}(\mathfrak{c}_m))) \in (0,-\sigma_m]$, respectively. Moreover,  let $j_{r}, j_{\ell} \in \nt$ be such that  $\mathcal{G}_m^{j_{r}}(\psi_m(1))$ is the most right endpoint of the Funnel and $\mathcal{G}_m^{j_{\ell}}(\psi_m(1))$ is the most left endpoint of the Funnel. Similarly, let  $\widehat{j_{r}}, \widehat{j_{\ell}} \in \nt$ such that  $\mathcal{G}_m^{-\widehat{j_{\ell}}}(\psi_m(B_{m,g}(h(\mathfrak{c}_m))))$ is the most left endpoint of the Funnel and $\mathcal{G}_m^{-\widehat{j_{r}}}(\psi_m(B_{m,g}(h(\mathfrak{c}_m))))$ is the most right endpoint of the Funnel.\\
 
 Note that $i_{\kappa} \neq j_{\kappa}$ for $\kappa \in \{c,r,\ell\}$.  Next result follows from a small adaptation of \cite[Pages 206-207]{KT}.
 
 \begin{lemma}\label{lemma:estimatesparameterstubular} There exists a constant $D_5>0$ such that if $r(m)$ is sufficiently big then
 	\[
 	|2r(m)- \pi \varepsilon_m^{-1/2}| \leq D_5\,\varepsilon_m^{\frac{-1+\alpha}{2}}
 	 	\]
 	 		\[
 	 	\left| \frac{h(\varepsilon_m)}{\varepsilon_m} -1 \right| \leq D_5\,  \varepsilon_m^{\alpha/2}
 	 	\]
 	 	\[
 	 	|i_c - r(m)| \leq D_5\, \varepsilon_m^{\frac{-1+\alpha}{2}}
 	 	\]
 	 	\[
 	 	|\widehat{i_c} - r(m)| \leq D_5\,  \varepsilon_m^{\frac{-1+\alpha}{2}}
 	 	\]
 	 	\[
 	 	 | i_{\alpha} - j_{\alpha} |  \leq  D_5\, \varepsilon_m^{\frac{-1+\alpha}{2}}
 	 	\]
 	 	\[
 	 	| \widehat{i_{\alpha}} - \widehat{j_{\alpha}} |  \leq  D_5\, \varepsilon_m^{\frac{-1+\alpha}{2}},
 	 	\]
 	 	for $\alpha \in \{c,r,\ell\}$.
 \end{lemma}
We also have a similar result when $\ell(m)$ is sufficiently big.

\section{Key estimates for points in $\widehat{E}_{n+1}$}\label{sec:key}
 
For $m \in \nt$ as in the previous section  let $n \geq m$.  In this section we prove that every endpoint $z  \in \widehat{E}_{n+1} \subset J_m$ is exponentially close to its image  $h(z)$.

Let us observe that if we assume  exponential convergence of the renormalizations (in the $C^{1}$-topology) then by \cite[Remark 5.5]{EG} there exists $C>0$ such that  
\begin{equation}\label{eq:diferenciafreecriticalpoint}
	\left|B_{n,f}(\mathfrak{c}_{n}^f) - B_{n,g}(\mathfrak{c}_{n}^g) \right|
	%&\leq |DB_{n,f}(z)| |\mathfrak{c}_{n}^f-\mathfrak{c}_{n}^g|
	%	+ \left| \frac{\mathfrak{c}_{n}^g-c_0}{|I_{n+1}^f|} - \frac{\mathfrak{c}_{n}^g-c_0}{|I_{n+1}^g|} \right|\\
	%	&\leq \frac{C\, |J_{n}|}{|I_{n}^f|} \, L\,  \lambda^{n/d} + 
	%	\frac{|\mathfrak{c}_{n}^g-c_0|}{|I_{n+1}^g|} \left| \frac{|I_{n+1}^g|}{|I_{n+1}^f|} -1\right|\\
	\leq  C\,\lambda^{n/d}.
\end{equation}

We begin with the following result.
 
%%%%%%%%%%%%%%%%%%%%%%%%%%%%%%%%%%%%%%%%%%%%

% \subsection{Estimates over the points in $\widehat{E}_{n}$ and their corresponding images by $h$} \label{subsec:estimates}
	
   	\begin{lemma}\label{lema_pasocero} There exist $C_1 > 1$ ( depending on real bounds and on $C$) and $\lambda_1=\lambda_1(\lambda, d) \in (0,1)$ such that for all $m<n+1$ we have:
   		\begin{equation*}
   		\max_{v  \in \widehat{E}_{n+1} \cap J_m} \left| B_{m,g} (h(v))- B_{m,f} (v)\right| \leq C_1 \left(\max_{v  \in \widehat{E}_{n+1} \cap J_{m+1}} \left| B_{m+1,g} (h(v))- B_{m+1,f} (v)\right| + \lambda_1^{m/2} \right).
   		\end{equation*}
   	\end{lemma}

   	\begin{proof}[Proof of Lemma \ref{lema_pasocero}]
   	Let $m \leq n$, then by the relation $B_{m,f}(t)= -B_{m+1,f}(t) \mathcal{R}^{m}f(0)$, for any $t \in J_{m+1}$, real bounds and the exponential converge of the renormalizations,  we obtain for any $v  \in \widehat{E}_{m+1} \subset J_{m+1}$ (in particular for $v=f^{q_{m+2}}(c_0)$) 
   	\begin{align*}\label{ine:estimatesextremesqn}
   		\left|B_{m,f}(v) - B_{m,g}(h(v)) \right| &=
   		\left|-B_{m+1,f}(v) \mathcal{R}^{m}f(0) + B_{m+1,g}(h(v)) \mathcal{R}^{m}g(0) \right|\\
   		&\leq \left|\mathcal{R}^{m}f(0) (B_{m+1,g}(h(v)) - B_{m+1,f}(v)) \right| + | B_{m+1,g}(h(v))| \left|\mathcal{R}^{m}g(0)-\mathcal{R}^{m}f(0) \right| \\   			
   		&\leq  C \, \left|B_{m+1,f}(v)- B_{m+1,g}(h(v))\right| + C\, \lambda^m  \\
   		&\leq C \left(  \max_{z \in \widehat{E}_{n+1} \cap J_{m+1}} \left|B_{m+1,g}(h(v)) - B_{m+1,f}(v)\right| + \lambda^m  \right).
   	\end{align*}
So we just need to prove the result for $v  \in \widehat{E}_{n+1} \cap (J_m \setminus J_{m+1})$ and $m< n+1$. Let us first prove for $v  \in \widehat{E}_{m+1} \cap (J_m \setminus J_{m+1})$. If $v=f^{q_m}(c_0)$  then $v$ satisfies the lemma since
	\begin{align*}
		|B_{m,f}(v) - B_{m,g}(h(v)) |
		&= \left| \frac{-B_{m-1,f}(v) }{\mathcal{R}^{m-1}f(0)} + \frac{B_{m-1,g}(h(v))}{ \mathcal{R}^{m-1}g(0)}  \right|\\
		&\leq \frac{|B_{m-1,g}(h(v)) - B_{m-1,f}(v) |}{|\mathcal{R}^{m-1}f(0)|} +|B_{m-1,g}(h(v))| \left| \frac{1}{\mathcal{R}^{m-1}g(0)} - \frac{1}{\mathcal{R}^{m-1}f(0)}\right|\\
		&\leq \frac{|I_m^f(c_0)|}{|I_{m+1}^f(c_0)|} \left|
		\frac{|I_m^g(c_0)|}{|I_{m-1}^g(c_0)|} - 	\frac{|I_m^f(c_0)|}{|I_{m-1}^f(c_0)|} \right| + 	\frac{|I_m^g(c_0)|}{|I_{m-1}^g(c_0)|} 
		\left| 	\frac{|I_{m-1}^f(c_0)|}{|I_{m}^f(c_0)|} - 	\frac{|I_{m-1}^g(c_0)|}{|I_{m}^g(c_0)|} \right| \\
		&\leq C^2 \, \lambda^{m-1}+ C^4 \, \lambda^{m-1}\\
		&\leq 2C^4 \lambda \left( \lambda^m + \max_{v \in \widehat{E}_{n+1} \cap J_m} |B_{m+1,g}(h(v)) - B_{m+1,f}(v)| \right).
\end{align*}
   Moreover, Inequality \eqref{eq:diferenciafreecriticalpoint} implies that the point $\mathfrak{c}_{m} \in \widehat{E}_{n+1}$ also satisfies the lemma for $\lambda_1=\max\{\lambda, \lambda^{1/d}\}$.
   Let $M>0$ given by Theorem \ref{thm:Mcontrolledmcm} and let $L=L(M)>0$ given by Lemma \ref{lemma:tubularcoordinates}. First, let us consider the endpoints  in $\widehat{E}_{m+1} \cap [f^{-q_{m+1}}(\mathfrak{c}_{m}), f^{q_m}(c_0)]$. Let us take $L$ forward iterates of the $\mathcal{R}^mf$ at $B_{m,f}(f^{q_m}(c_0))$ and $L$ backward iterates of $\mathcal{R}^mf$ at $B_{m,f}(f^{-q_{m+1}}(\mathfrak{c}_{m}))$. If $r(m) \leq 2L$, then we cover all the endpoints and for them  we obtain the lemma  (for $C_1=C^{L}>1$ and $\lambda_1=\max\{\lambda, \lambda^{1/d}\}$).
    Now,  if $r(m) > 2L$, then we enter to the tubular $M_{m,2L}^f$ precisely at  $B_{m,f}(f^{(L+1)q_{m+1}+q_m}(c_0))$ and at $B_{m,f}(f^{-(L+2)q_{m+1}}(\mathfrak{c}_{m}))$, and the center $z_m$ is defined. 
   Then we take $P=\lceil \lambda^{-m/2} \rceil$ forward iterates of  $\mathcal{R}^mf$ at $B_{m,f}(f^{(L+1)q_{m+1}+q_m}(c_0))$ and $P$ backward iterates of $\mathcal{R}^mf$ at $B_{m,f}(f^{-(L+2)q_{m+1}}(\mathfrak{c}_{m}))$.\\
    If $z_m \notin B_{m,f}([f^{-(L+1)q_{m+1}}(\mathfrak{c}_m), f^{-(L+1+P)q_{m+1}}(\mathfrak{c}_m)]) \cap  B_{m,f}([f^{q_m+(L+P)q_{m+1}}(c_0), f^{q_m+Lq_{m+1}}(c_0)])$, then since $D\mathcal{R}^mf(x)<1$ in $B_{m,f}([f^{q_m+(L+P)q_{m+1}}(c_0), f^{q_m+Lq_{m+1}}(c_0)])$ we have, for $1 \leq j \leq P$ and $\lambda_1=\max\{\lambda, \lambda^{1/d}\}$
   {\small
   \begin{align*}
   	|B_{m,f}(f^{(L+j)q_{m+1}+q_m}(c_0)) - B_{m,g}(g^{(L+j)q_{m+1}+q_m}(c_0)) |
   	  	&\leq P\,C_1\, \left( \lambda_1^{m/2} + \max_{v \in \widehat{E}_{n+1} \cap J_m} |B_{m+1,g}(h(v)) - B_{m+1,f}(v)| \right),
   \end{align*}
}
and since $D\mathcal{R}^mf(x)>1$ in $B_{m,f}([f^{-(L+1)q_{m+1}}(\mathfrak{c}_m), f^{-(L+1+P)q_{m+1}}(\mathfrak{c}_m)])$ then, for $1 \leq j \leq P$, we have
\begin{multline*}
	|B_{m,f}(f^{-(L+1+j)q_{m+1}}(\mathfrak{c}_m)) - B_{m,g}(g^{-(L+1+j)q_{m+1}}(h(\mathfrak{c}_m))) |
	\\ \leq P\,C_1\, \left( \lambda_1^{m/2} + \max_{v \in \widehat{E}_{n+1} \cap J_m} |B_{m+1,g}(h(v)) - B_{m+1,f}(v)| \right).
\end{multline*}
 %%%%%%%%%%%%%%%%%%%%%%%%%%
 For the rest of endpoints $v \in \widehat{E}_{m+1}$ we have the following claim. Note that  if we reach or pass the center $z_m$ at some previous iterate (less than $P$) then we stop at that iterate and we also have the claim for the rest of endpoints.
 
 \begin{claim}\label{claimP}
 	Let  $D_1$ as in Lemma \ref{lemaKTfunnel} and $D_4$ as in Lemma \ref{lemaKTtunnel}. There exists a constant $C_*=C_*(C_0,D_1,D_4,M)>1$ and $\lambda_* =\lambda_*(\lambda, \alpha) \in (0,1)$ such that
 	\[
 		|B_{m,f}(v)-z_{m}|< C_*\, \lambda_*^{m/2}.
 	\]
 \end{claim}
  
\begin{proof}[Proof of Claim \ref{claimP}]
Let $\mathcal{F}_m$ be the tubular-coordinates  of $\mathcal{R}^mf$ with respect to $z_{m}$ (see Definition \ref{def:tubularcoordinates}). We want to  estimate $|\phi_m(B_{m,f}(v))|$ for $v \in  [ B_{m,f}(f^{-(L+P+1)q_{m+1}}(\mathfrak{c}_m)), B_{m,f}(f^{(L+P)q_{m+1}+q_m}(c_0))]$. % and for $v\in [B_{m,f}(f^{-(L+P)q_{m+1}+q_{m+2}}(c_0)), B_{m,f}(f^{(L+P)q_{m+1}}(\mathfrak{c}_m))]$ if $\ell(m)>2(L+P)$. Let us give the proof for the case $r(m)>2(L+P)$, the proof in the other case is analogous. 
Since $r(m) \, (z_{m}-\mathcal{R}^mf(z_{m})) <1$ then by Theorem \ref{thm:Mcontrolledmcm} we have
\begin{equation}\label{eq:varepsiloneM}
0< |\varepsilon_m| = |\mathcal{F}_m(0)|
< \frac{|D^2\mathcal{R}^mf(z_{m})|}{2\, r(m)} <  \frac{M}{4(L+P)} < \frac{M}{4P} < \frac{M}{4} \, \lambda^{m/2}.
\end{equation} 
If $\phi_m(B_{m,f}(v))$ is inside the Tunnel then we already have the claim for that point since
\begin{align*}
	|\phi_m(B_{m,f}(v))| \leq C_0 \, |\varepsilon_m|^{\frac{1}{2+\alpha}} %\leq \frac{C_0 \,M^{1/3}}{4^{1/3}(L+P)^{1/3}} 
	\leq C_0\,\left(\frac{ M}{4}\right)^{\frac{1}{2+\alpha}} \, \left(\lambda^{\frac{1}{2+\alpha}}\right)^{m/2},
\end{align*}
and hence
\begin{equation*}
	|B_{m,f}(v)-z_m| \leq  \frac{2C_0}{M} \left(\frac{ M}{4}\right)^{\frac{1}{2+\alpha}} \, \left(\lambda^{\frac{1}{2+\alpha}}\right)^{m/2}.
\end{equation*} 
Otherwise, if $\phi_m(B_{m,f}(v))$ is inside the Funnel, then 
$B_{m,f}(v)$ corresponds to some iterate $\mathcal{F}_{m}^i(\phi(1))$ or $\mathcal{F}_m^{-i}(\phi(B_{m,f}(\mathfrak{c}_m)))$ for some $i$.
In the first case, using inequalities \eqref{eq:varepsiloneM} and  \eqref{eq_funnelsis0}  for the map $\mathcal{F}_m$, with $s_0=\mathcal{F}_m^{i_c}(\phi(1))$ and $s_k=\mathcal{F}_m^{i_c+k}(\phi(1))$,  we get, for all $k \geq 0$
\begin{align*}
	|\mathcal{F}_m^{i_{\ell}+k}(\phi(1))| =|s_{i_{\ell}-i_{c}+k}|  &\leq
	\left|s_{i_{\ell}-i_{c}+k}- \frac{1}{i_{\ell}-i_{c}+k+ s_0^{-1} }\right| + \left|\frac{1}{i_{\ell}-i_{c}+k+ s_0^{-1}} \right|\\ 
	&\leq D_1 \left(\frac{s_0}{s_0(i_{\ell}-i_{c}+k)+1}\right)^{2} + \frac{s_0}{s_0(i_{\ell}-i_{c}+k)+1} \\
	&\leq (D_1+1)\,s_0\\	
	&\leq   (D_1+1)\, |\varepsilon_m| \\
	&\leq \frac{M\,(D_1+1)}{4} \, \lambda^{m/2}.
\end{align*}
Since $\left |\frac{1}{2} (D^2\mathcal{R}^mf(z_{m}) \, [(\mathcal{R}^ m_f)^{j}(B_{m,f}(f^{q_m}(c_0)))- z_{m}] \right|=|\mathcal{F}_m^{j}(\phi(1))|$ then
\[
\left|B_{m,f}(f^{(i_{\ell}+k)q_{m+1}+ q_m}(c_0))- z_{m} \right| = \left|(\mathcal{R}^m_f)^{(i_{\ell} +k)}(B_{m,f}(f^{q_m}(c_0)))- z_{m}\right| %= \frac{2 \, |\mathcal{F}^{i_{\ell}+ k}(\phi(1))| }{|D^2\mathcal{R}^mf(z_{m})|}
%&\leq \frac{M^2\,(D_1+1)}{2 (L+P)} \\
%&\leq \frac{M^2\,(D_1+1)}{2P}  \\
\leq \frac{M^2\,(D_1+1)}{2} \, \lambda^{m/2}.
\]
We obtain the claim for all those points above choosing $C_*=\max \{C_0M^{-1/2}, M(D_1+1)/2\}>1$ and $\lambda_*=\max\{\lambda, \lambda^{1/2+\alpha} \}$. Since any endpoint $v \in \widehat{E}_{m+1}$ that can be written as $f^{-kq_{m+1}}(\mathfrak{c}_m)$, for $1 \leq k \leq r(m)$, is between consecutive points of the form $f^{(i_{\ell}+t)q_{m+1}+ q_m}(c_0)$,  then we have for all $k \geq 1$
\begin{align*}
	\left|B_{m,f}(f^{-kq_{m+1}}(\mathfrak{c}_m))- z_{m} \right| \leq C_* \lambda_*^{m/2},
\end{align*} 
and we get the claim for any endpoint $v \in \widehat{E}_{m+1}$ such that $\phi_m(B_{m,f}(v))$ is inside the left part of the Funnel. Moreover, since any endpoint $v \in \widehat{E}_{m+1}$ which is inside the right part of the Funnel is between consecutive points of the form $f^{kq_{m+1}+ q_m}(c_0)$ or $f^{-kq_{m+1}}(\mathfrak{c}_m)$, then we get the claim also for all these points.
%Now, using Lemma \ref{lemaKTfunnel} for the map $\mathcal{F}^{-1}$ and the same argument as above we get the claim for any point of the form $B_{m,f}(f^{-kq_{m+1}}(\mathfrak{c}_m))$ inside the right side of the Funnel. Since any endpoint in $\widehat{E}_{n+1}$ inside the right side of the Funnel is between two consecutive points of the form  $B_{m,f}(f^{iq_{m+1}+ q_m}(c_0))$, then we obtain the claim for  $B_{m,f}(v)$ such that $\phi(B_{m,f}(v))$ is inside the right side of the Funnel. 
%|\mathcal{F}^{-1}(y)+ \varepsilon -y-y^2| \leq \varepsilon(y-\varepsilon)+ C|y-\varepsilon|^3
\end{proof}

%%%%%%%%%%%%%%%%%%%%%%%%%
   By Claim \ref{claimP} and Inequality \eqref{inq:estimativacentrofunnel}, any endpoint $v \in \widehat{E}_{m+1} \cap I_m$ such that $\phi(B_{m,f}(v))$ is inside the Funnel or the Tunnel of $M_{m,2L}^f$, satisfies
 \begin{align*}
 	| B_{m,f}(v) - B_{m,g}(h(v))|  	&\leq |z_{m} - B_{m,f}(v)| +|z_{m} - h(z_{m})| + |h(z_{m}) - B_{m,g}(h(v)) | \\
 	&\leq  2C_*\, \lambda_*^{m/2} + C\lambda^{m} \\
 	&\leq C_1\, \left( \lambda_*^{m/2} + \max_{v \in \widehat{E}_{m+1} \cap J_m} |B_{m+1,g}(h(v)) - B_{m+1,f}(v)| \right).
 \end{align*}
Now, for the endpoints $v \in \widehat{E}_{m+2}$, we observe the following. If $v \in \widehat{E}_{m+2} \cap I_{m+1}$ then we can repeat the proof that we made above. If  $v \in \widehat{E}_{m+2} \cap f^{q_{m}}(I_{m+1})$ then we also obtain the lemma since the maps $f^{q_m}$ and $g^{q_m}$ are exponentially close inside $I_{m+1}$. If $v \in \widehat{E}_{m+2} \cap [\mathfrak{c}_m, f^{-q_{m+1}}(\mathfrak{c}_m)]$, then since the maps  $f^{q_{m+2}}$ and $g^{q_{m+2}}$ inside $ [\mathfrak{c}_m, f^{-q_{m+1}}(\mathfrak{c}_m)]$ are exponentially close, we can repeat the proof above inside that interval replacing $m+1$ by $m+2$. The proof for the others endpoints follows from the  fact that the two maps $f^{q_{m+1}}$ and $g^{q_{m+1}}$ are  exponentially close and that any of the resting endpoints is  image by iterates of $f^{q_{m+1}}$ (that we have controlled in the proof above) at an endpoint in $\widehat{E}_{m+2}$ that already satisfies the lemma.
Consequently, we obtain Lemma \ref{lema_pasocero}, with $C_1= 2\max\{2C_*, C\}>1$ and $\lambda_1= \max\{ \lambda_*, \lambda^{1/d}\} \in (0,1)$, for any endpoint in $\widehat{E}_{n+1} \cap J_m$. 
%for $M+1 \leq i \leq r(m)$,
%{\small
%	\begin{align*}
%		| &B_{m,f}(f^{-kq_{m+2}}(c_0)) - B_{m,g}(g^{-kq_{m+2}}(c_0))| 
%		&\leq C\, \lambda^m,
%	\end{align*}
%}
%and $| B_{m,f}(f^{-jq_{m+1}}(\mathfrak{c}_{m}^f)) - B_{m,g}(g^{-jq_{m+1}}(\mathfrak{c}_{m}^g))| 
%		\leq C\, \lambda^m$, \ for $-r(m) \leq j \leq -2$.
%and for $-r(m) \leq j  \leq 0$,
%{\small
%\begin{align*}
%	| B_{m,f}(f^{-jq_{m+1}}(\mathfrak{c}_{m}^f)) &- B_{m,g}(g^{-jq_{m+1}}(\mathfrak{c}_{m}^g))| \\ &\leq S_1\, \left( \max_{v \in \widehat{E}_{n+1} \cap J_m} |B_{m+1,g}(\mathfrak{c}_{m}^g) - B_{m+1,f}(\mathfrak{c}_{m}^f)| \right) + C\, \lambda^{m/d}.
%\end{align*}
%}
\end{proof}

\begin{remark}\label{remark:C1grande}
	Let us observe that in the previous lemma we can take $C _1$ such that $C_1 \lambda_1>2$.
\end{remark}

\begin{proposition}\label{pro_pasouno} Let $\lambda_1\in (0,1)$ and $C_1>1$  as in Lemma \ref{lema_pasocero}. For any $\lambda_2 \in (\lambda_1, 1)$ there exist $b_1=b_1(C,C_1, \lambda_2) \in (0, 1)$ and $C_2=C_2(C,C_1, \lambda_1) > 1$
	such that
	\begin{equation*}
		| B_{m,g} (h(v))- B_{m,f} (v)| \leq C_2\, \lambda_2^n
	\end{equation*}
	for  $(1-b_1)n \leq m \leq n+1$ and $v \in \widehat{E}_{n+1} \cap J_{m}$.
\end{proposition}

\begin{proof}[Proof of Proposition \ref{pro_pasouno}]
	Let $m=n+1$. Note that the set $\widehat{E}_{n+1} \cap J_{n+1}$ is equal to
	$\{f^{q_{n+1}}(c_0), f^{q_{n+2}}(c_0), \mathfrak{c}_{n}^f\}$ or $\{f^{q_{n+1}}(c_0), f^{q_{n+2}}(c_0)\}$  depending on whether $m$ is or is not a two-bridges level. We observe that (by definition)  $| B_{n+1,f}(f^{q_{n+1}}(c_0)) - B_{n+1,g}(g^{q_{n+1}}(c_0))| =0$, Inequality  \eqref{eq:diferenciafreecriticalpoint} implies that $|B_{n+1,f}(\mathfrak{c}_{n}^f) - B_{n+1,g}(\mathfrak{c}_{n}^g)| \leq C\, \lambda^{(n+1)/d}$ and that the hypotheses under the renormalizations (Inequality \eqref{ineq:convergencerenormalizations})  implies that $|B_{n+1,f}(f^{q_{n+2}}(c_0)) - B_{n+1,g}(g^{q_{n+2}}(c_0))|  %\left| \frac{|I_{n+2}^f|}{|I_{n+1}^f|} - \frac{|I_{n+2}^g|}{|I_{n+1}^g|} \right|
	 \leq C\, \lambda^{n+1}$.
Therefore,$$\max_{v \in \widehat{E}_{n+1} \cap J_{n+1}} |B_{n+1,f}(v) - B_{n+1,g}(h(v))| \leq C\ (\lambda^{1/d})^{n+1}  \leq C\, \lambda_1^{(n+1)}.$$If $m < n+1$ then by previous inequality and Lemma \ref{lema_pasocero} (inductively) we get for $\widetilde{C}=\max\{ C, C_1\}>1$
	\begin{align*}
	\max_{v \in \widehat{E}_{n+1} \cap J_{m}} |B_{m,f}(v)  - B_{m,g}(h(v))|
	&\leq \sum_{k=0}^{n-m-1} \, \widetilde{C}^{n-m-k+2} \lambda_1^{n-k+1}  \\
	&= \lambda_1^{n+1}\, \widetilde{C}^{n-m+2}\, \sum_{k=0}^{n-m-1} \, \frac{1}{(\widetilde{C} \lambda_1)^{k}} \\
	 &=  \lambda_1^{n+1}\, \widetilde{C}^{n-m+2}\, 
	 \left( \dfrac{(1/\widetilde{C}\,\lambda_1)^{n-m}-1}{(1/\widetilde{C}\,\lambda_1)-1} \right) \\
%	&= \frac{\lambda_1^{n+2}\, C^{n-m+3}\, \left( (C\,\lambda_1)^{n-m}-1\right) }{
%	(C\,\lambda_1)^{n-m}(C\,\lambda_1-1)}\\
&\leq \frac{\widetilde{C}^{n-m+3} \, \lambda_1^{n+2}}{\widetilde{C}\lambda_1-1}\\
%	&\leq \sum_{k=0}^{n-m-+} \, C^{n-m-k+2} \lambda_1^{n-k+1}  \\
  	%&\leq \lambda_1^{n+1}\, C^{n-m+2}\, \sum_{k=0}^{n-m-1} \, \frac{1}{(C \lambda_1)^{k}} % \\
  	%	 &\leq \sum_{j=0}^{n-m-1} \, C^{j+1} \lambda^{(m+j)/2} + C^{n-m}  \lambda^n \\
  	&\leq \left(\frac{\lambda_1^2 \ \widetilde{C}^3}{\widetilde{C} \lambda_1 -1}\right) (\widetilde{C}^{b_1}\lambda_1)^n.
  	\end{align*}
  By Remark \ref{remark:C1grande}, $\widetilde{C} \lambda_1>2$. The result follows taking $C_2=\lambda_1^2 \ \widetilde{C}^3 (\widetilde{C} \lambda_1 - 1)^{-1}>1$ and, given $\lambda_2 \in (\lambda_1,1)$, choosing $b_1 \in (0, 1)$ such that  $ \widetilde{C}^{b_1}\lambda_1 < \lambda_2$.
	\end{proof}

\section{Key estimates for intervals of $\mathcal{\widehat{P}}_{n+1}$}\label{sec:estimatesintervals}

In this section we will extend the estimate obtained in Proposition \ref{pro_pasouno}, but this time for intervals of $\widehat{\mathcal{P}}_{n+1}$ inside $J_m$ with $(1-b_1)n \leq m < n-1$. First, using an estimate for adjacent intervals, we get the estimates for intervals of $\widehat{\mathcal{P}}_{n+1}$ inside $J_{n}$. Then, by an inductive argument, we get the estimates for intervals inside $J_m$.

Let us first get an estimate for an interval of $\widehat{\mathcal{P}}_{n+1}$ and its image by the first return map $f^{q_{m+1}}$ for $(1-b_1)n \leq m \leq n$. As we mention before, in general, if $I \in  \widehat{\mathcal{P}}_{n+1}$ then not necessarily $f^{q_{m+1}}(I) \in \widehat{\mathcal{P}}_{n+1}$.

 \begin{proposition}\label{Pro_pasodos} Let $\lambda_2 \in (0,1)$, $C_2>1$ and $b_1 \in (0,1)$ be as in Proposition \ref{pro_pasouno}. There exist $C_3=C_3(\lambda_2, C_2,M)>1$ and $\lambda_3=\lambda_3(\lambda_2, b_1) \in (0,1)$  such that for any  $(1-b_1)n \leq m \leq n$ and  $I \in \widehat{\mathcal{P}}_{n+1}$, with  $I \subset (I_m \setminus I_{m+2})$,  we have:
	\begin{equation*}
		\left| \log \frac{|B_{m,g}(h(f^{q_{m+1}}(I))|}{|B_{m,f}(f^{q_{m+1}}(I))|} -  \log \frac{|B_{m,g}(h(I))|}{|B_{m,f}(I)|}  \right| \leq C_3 \, \lambda_2^{n/2}.
	\end{equation*}
\end{proposition}

\begin{proof}[Proof of Proposition \ref{Pro_pasodos}]
	%Note that for $m=n$, the proof follows from  Lemma \ref{lemma:derivadameanvaluethrm}. 
	Let $I=[v,w]$ with $v,w \in \widehat{E}_{n+1}$ and let us assume that $I$ does not have $\mathfrak{c}_m$ as an endpoint.
	\begin{itemize}
		\item [a)] If $|B_{m,f}(I)|=|B_{m,f}(v) - B_{m,f}(w)| \geq \lambda_2^{n/2}$ then by Proposition \ref{pro_pasouno} we have, for $(1-b_1)n \leq m \leq n+1$,
			\begin{align*}
				\left| \frac{B_{m,g}(h(I))}{B_{m,f}(I)} -1 \right|
				%&=  \left| \frac{B_{m+1,g}(h(I)) - B_{m+1,f}(I)}{B_{m+1,f}(I)} \right| \\
				\leq  \frac{|B_{m,g}(h(v))-B_{m,f}(v)| + |B_{m,g}(h(w))-B_{m,f}(w)|}{\lambda_2^{n/2}} 
				\leq \frac{2\, C_2\, \lambda_2^{n}}{\lambda_2^{n/2}} = 2\,C_2 \, \lambda_2^{n/2}.
			\end{align*}
		On the other hand, by Mean Value Theorem, $M$-controlled hypothesis (Theorem \ref{thm:Mcontrolledmcm}) and the assumption  we have 
		\begin{align*}
			|B_{m,f}(f^{q_{m+1}}(I))| % &=
			%	|B_{m+1,f}(f^{q_{m+1}}(v)) - B_{m+1,f}(f^{q_{m+1}}(w))|\\ 
			%	&= 	|\mathcal{R}^m_f(B_{m+1,f}(v)) - \mathcal{R}^m_f(B_{m+1,f}(w))| \\
			&= |D\mathcal{R}^m_f(z^*)|\, |B_{m,f}(v) - B_{m,f}(w)|  \geq \frac{1}{M} \lambda_2^{n/2}.
		\end{align*}	
	Therefore, using  previous inequality, Proposition  \ref{pro_pasouno} and Mean Value Theorem we obtain,
			\begin{align*}
			\hspace{1cm}	\left| \frac{B_{m,g}(g^{q_{m+1}}(h(I)))}{B_{m,f}(f^{q_{m+1}}(I))} -1 \right| 
				%&=  \left| \frac{B_{m+1,g}(h(I)) - B_{m+1,f}(I)}{B_{m+1,f}(I)} \right| \\
				&\leq \frac{|\mathcal{R}^m_g(B_{m,g}(h(v))) - \mathcal{R}^m_g(B_{m,g}(h(w)))| \ + \ |\mathcal{R}^m_f(B_{m,f}(v)) - \mathcal{R}^m_f(B_{m,f}(w))|}{|B_{m,f}(f^{q_{m+1}}(I))|}\\
				&\leq  \frac{M^2\, \left( \ |B_{m,f}(w)-B_{m,g}(h(w))| + |B_{m,f}(v)-B_{m,g}(h(v))| \ \right)}{\lambda_2^{n/2}} \\
				&\leq 2 \, M^2 \, C_2\, \lambda_2^{n/2}.
			\end{align*}
		The proof follows in this case taking  $C_3=2M^2 C_2$.		
		\item [b)] If $|B_{m,f}(I)|=|B_{m,f}(v) - B_{m,f}(w)| \leq \lambda_2^{n/2}$, then by the Mean Value Theorem and the exponential converge of renormalizations there exist $a \in B_{m,g}(h(I))$, $b \in B_{m,f}(I)$ and $\beta$ between $a$ and $b$ (note that $\beta$ can not be a critical point of $\mathcal{R}^mf$) such that
		{\small
		\begin{align*}
		\hspace{1cm}	\left|  \log \frac{|B_{m,g}(h(f^{q_{m+1}}(I))|}{|B_{m,f}(f^{q_{m+1}}(I))|} - \log \frac{|B_{m,g}(h(I)|}{|B_{m,f}(I)|}  \right| 
			&= \left| \log \frac{|B_{m,g}(h(f^{q_{m+1}}(I))|}{|B_{m,g}(h(I))|} - \log \frac{|B_{m,f}(f^{q_{m+1}}(I))|}{|B_{m,f}(I)|} 			\right| \\
			&= \left| \log D\mathcal{R}^mg(a)  - \log  D\mathcal{R}^mf(b)  \right| \\
			&\leq \left| \log D\mathcal{R}^mg(a)  - \log  D\mathcal{R}^mf(a)  \right| +  \left| \log D\mathcal{R}^mf(a)  - \log D\mathcal{R}^mf(b)  \right| \\ 
			&\leq C\, M\, \lambda^m  + \frac{|D^2\mathcal{R}^mf(\beta)|}{|D\mathcal{R}^mf(\beta)|} |b-a|\\
			&\leq C\, M\, \lambda^m  + M^{2} |b-a|.
		\end{align*}
	}
		Since
			\begin{align*}
	 \hspace{0.8cm}	|b-a| \leq |B_{m,f}(I)| + |B_{m,g}(h(v))-B_{m,f}(v)| + |B_{m,g}(h(w))-B_{m,f}(w)| \leq \lambda_2^{n/2} + 2\, C_2\, \lambda_2^{n},
			\end{align*}
and  $\lambda^m \leq \lambda_2^{n/2}$, then we obtain the lemma  for  $C_3=\max\{ CM, \, 2M^2 \,C_2 \}>1$. %since $0<b_1<0$ and $$\lambda$\in (0.1) then $\lambda^{2(1-b_1) \in (0,1)$
%		\begin{equation*}
%\left| \ \log \frac{|B_{m,g}(h(f^{q_{m+1}}(I))|}{|B_{m,f}(f^{q_{m+1}}(I))|} - \log \frac{|B_{m,g}(h(I)|}{|B_{m,f}(I)|} \ \right| \\
%			\leq  K_3\, \lambda_3^{n/2}.
%		\end{equation*}
	\end{itemize}
Now, if $I$ has $\mathfrak{c}_m$ as an endpoint, we  obtain the proposition repeating the same proof above but for $f^{-q_{m+1}}$ instead of $f^{q_{m+1}}$ and for $I$ replacing  by $f^{-q_{m+1}}(I_L)$, where $I_L=f^{q_{m+1}}(I)$.
\end{proof}

Let us note that Proposition \ref{Pro_pasodos} and its proof can be extend to intervals $I \in \widehat{\mathcal{P}}_{n+1}$ with $I \subset I_{m+1}$, and for the map $f^{q_m}$ instead of $f^{q_{m+1}}$:

\begin{corollary}\label{coro:Pro_pasodos}
	 Let  $b_1 \in (0,1)$ be as in Proposition \ref{pro_pasouno},   $C_3>0$ and $\lambda_2 \in (0,1)$ as in Proposition \ref{Pro_pasodos}. For any  $(1-b_1)n \leq m \leq n$,  $I \in \widehat{\mathcal{P}}_{n+1}$ with  $I \subset  I_{m+1}$  we have:
		\begin{equation*}
			\left| \log \frac{|B_{m,g}(h(f^{q_{m}}(I))|}{|B_{m,f}(f^{q_{m}}(I))|} -  \log \frac{|B_{m,g}(h(I))|}{|B_{m,f}(I)|}  \right| \leq C_3 \, \lambda_2^{n/2}.
		\end{equation*}
\end{corollary}

Now, let us obtain an estimate  for every interval $I \in \widehat{\mathcal{P}}_{n+1}$ contained in $J_{n}$.

	\begin{proposition}\label{Pro_pasotres} There exist a universal constant $C_4=C_4(C, C_2,C_3)>0$ and  $\lambda_3 = \lambda_3(\lambda_2) \in (0,1)$ such that for any  $I \in \widehat{\mathcal{P}}_{n+1}$, with $I \subset J_{n}$, 
	\begin{equation*}
		\left| \log \frac{|h(I)|}{|I|} \right| \leq C_4 \, \lambda_3^{n}.
	\end{equation*}
\end{proposition}

\begin{proof}[Proof of Proposition \ref{Pro_pasotres}]
	Observe that for any $I \in \widehat{\mathcal{P}}_{n+1}$ we have:
	\begin{equation*}
		\left| \log \frac{|h(I)|}{|I|} \right| =
		\left| \log  \frac{|h(I_{n})|}{|I_{n}|} \, \frac{|B_{n,g} (h(I))|}{|B_{n,f}(I)|} \right| \leq
		\left| \log  \frac{|h(I_{n})|}{|I_{n}|}  \right| +
		\left| \log  \frac{|B_{n,g} (h(I))|}{|B_{n,f}(I)|} \right|.
	\end{equation*}
	By Inequality \eqref{eq:limitfundintervals}, we just need to prove that for any $I \in \widehat{\mathcal{P}}_{n+1} \setminus \{I_{n+1}, I_{n+2}\}$, contained in $I_n$, there exist $\widetilde{\lambda}= \widetilde{\lambda}(\lambda_2, \lambda_3) \in (0,1)$ and $\widetilde{C}>0$ such that
	\begin{equation}\label{eq:pasotres}
		\left| \log   \frac{|B_{n,g} (h(I))|}{|B_{n,f}(I)|} \right| \leq \widetilde{C}\, \widetilde{\lambda}^n.
	\end{equation}
Let us prove Inequality \eqref{eq:pasotres} for $I \in \left\{  \{ \widehat{\Delta}_n^{-j} \}_{j=1}^{j=r(n) - 1}, \, \Delta_R^{n+1},  \{ \Delta_{1}^{j} \}_{j=0}^{j=r(n)-1} 
\right\}$, the proof is analogous for intervals in the set $\left\{I_{n+2}(c_0), \{ \Delta_{a_{n+1}}^{j} \}_{j=0}^{j=\ell(n) - 1}, \, \Delta_L^{n+1}, \{ \widehat{\Delta}_n^j \}_{j=0}^{j=\ell(n)}\right\}$. By Proposition \ref{pro_pasouno} and the fact that both lengths $|\widehat{\Delta}_n^{-1}|$ and $|\Delta_1|$ are comparable with $J_n$, we have that both intervals, $\Delta_1$ and $\widehat{\Delta}_n^{-1}$, satisfy inequality \eqref{eq:pasotres} for $\widetilde{C}=2C_2 C$ and $  \widetilde{\lambda} = \lambda_2$. Let $M>0$ given by Theorem \ref{thm:Mcontrolledmcm} and let $L=L(M)>0$ given by Lemma \ref{lemma:tubularcoordinates}. By Proposition \ref{Pro_pasodos} (with $m=n$ for $I=\Delta_1$ and  $I=f^{-q_{n+1}}(\widehat{\Delta}_n^{-1})$), iterating the map $\mathcal{R}^nf$ exactly $L$ times  forward from $\Delta_1$ and  $L$ times backward from $\widehat{\Delta}_n^{-1}$, we have for $1 \leq j \leq L$,
	\begin{equation*}
		\left| \log   \frac{|B_{n,g} (h(\Delta_1^j))|}{|B_{n,f}(\Delta_1^j)|} \right| \leq 2(L+1) C_3 C\, \lambda_2^n
		\text{ \ \ and \ }
		\left| \log   \frac{|B_{n,g} (h(\widehat{\Delta}_n^{-j}))|}{|B_{n,f}(\widehat{\Delta}_n^{-j})|} \right| \leq  2(L+1) C_3 C\, \lambda_2^n .
	\end{equation*}
 If $r(n) \leq 2L$ the we get the lemma for each $I \in \left\{ \{ \widehat{\Delta}_n^{-j} \}_{j=2}^{j=r(n) - 1}, \,  \{ \Delta_{1}^{j} \}_{j=1}^{j=r(n)-1} \right\}$ and we are done. 
If $r(n)>2L$ then we iterate $P=\lceil{\lambda_2^{-n/2}}\rceil -2$ times (forward and also backward) and by Proposition \ref{Pro_pasodos} we obtain, for  $1 \leq j \leq L + P$ 
\begin{equation}\label{eq:iniciointervalosn}
	\left| \log   \frac{|B_{n,g} (h(\Delta_1^j))|}{|B_{n,f}(\Delta_1^j)|} \right| \leq 2(L+1) C_3 C \, \lambda_2^{n/2}
 \text{ \ \ and \ }
	\left| \log   \frac{|B_{n,g} (h(\widehat{\Delta}_n^{-j}))|}{|B_{n,f}(\widehat{\Delta}_n^{-j})|} \right| \leq 2(L+1) C_3 C\, \lambda_2^{n/2}.
\end{equation}
 If $r(n)\leq 2( L+ P)$, then we get the lemma for $I \in \left\{ \{ \widehat{\Delta}_n^{-j} \}_{j=2}^{j=r(n) - 1}, \,  \{ \Delta_{1}^{j} \}_{j=1}^{j=r(n)-1} \right\}$ and we are done.

%\begin{align*}
%	&\left|  \frac{|B_{n,g} (h(\Delta_1))|}{|B_{n,f}(\Delta_1)|} -1 \right| \\
%	&\hspace{0.7cm} \leq 	\frac{|B_{n,g}(g^{q_n}(c_0))- B_{n,f}(f^{q_n}(c_0))| +|B_{n,g}(g^{q_{n+1}+q_n}(c_0))- B_{n,f}(f^{q_{n+1}+q_n}(c_0))|}{|B_{n,f}(\Delta_1)|} \\
%	&\hspace{0.7cm}\leq C\, S_2\, \lambda_2^n
%	\end{align*}
%and
%\begin{equation}\label{eq:Delta1paso3}
	%\left| \log   \frac{|B_{n,g} %(h(\Delta_{1}))|}{|B_{n,f}(\Delta_{1})|} \right| =
	%\left| \log  \frac{|h(\Delta_{1})|}{|h(I_{n})|} %\frac{|I_{n}|}{|\Delta_1|} \right| \leq 
%	\left| \log  \frac{|h(\Delta_{1})|}{|\Delta_{1}|}  \right|	+
	%\left| \log  \frac{|h(I_{n})|}{|I_{n}|} \right|,
%\end{equation}
%\begin{align*}
%\left|  \frac{|B_{n,g} (h(\widehat{\Delta}_n^{-1}))|}{|B_{n,f}(\widehat{\Delta}_n^{-1})|} -1 \right| 
%&\leq 	\frac{|B_{n,g}(\mathfrak{c}_{n}^g)- %B_{n,f}(\mathfrak{c}_{n}^f)| %+|B_{n,g}(g^{-q_{n+1}}(\mathfrak{c}_{n}^g))- %B_{n,f}(f^{-q_{n+1}}(\mathfrak{c}_{n}^f))|}{|B_{n,f}(\widehat{\D%elta}_n^{-1})|} \\
%	&\leq C\, S_2\, \lambda_2^n. 
%\end{align*}
%or, equivalently,
%{\small
%\begin{equation}\label{eq:Deltan-1paso3}
%	\left| \log \frac{|B_{n,g} (h(\widehat{\Delta}_n^{-1}))|}{|B_{n,f}(\widehat{\Delta}_n^{-1})|} \right| \leq  C\, \lambda_2^n.
%\end{equation}
%}

 Otherwise, if $r(n) > 2(L+P)$ then the intervals $\Delta_1^j$ and $\widehat{\Delta}_n^{-j}$, for  $L +P \leq j$, are all inside the tubular $M_{n,2L}^f$ whose center $z_n$ is defined. 
 At this point we change of coordinates and consider intervals whose endpoints are given by  the tubular-coordinates (see Definition \ref{def:tubularcoordinates}), i.e. let $\mathcal{F}_n=\varphi \circ \mathcal{R}^n_f \circ \varphi^{-1}$, where $\varphi(x)=D^{2}\mathcal{R}^n_f(z_n)(x-z_n)/2$, and take $|\mathcal{I}_j|=|\mathcal{F}_n^{j+1}(\varphi(1))-\mathcal{F}_n^{j}(\varphi(1))|$ and $|\mathcal{J}_j| =| \mathcal{F}_n^{-j}(\phi(B_{n,f}(\mathfrak{c}_n))-
 \mathcal{F}_n^{-j-1}(\phi(B_{n,f}(\mathfrak{c}_n))|$. We also consider $\mathcal{G}_n=\psi \circ \mathcal{R}^n_g \circ \psi^{-1}$, where $\psi(x)=D^{2}\mathcal{R}^n_g(h(z_n))(x-h(z_n))/2$.
 \begin{claim}\label{claimI}
 For $i \geq L+P$ we have
  		\begin{equation*}
 			\left| \log  \frac{|B_{n,g} (h(\Delta_1^i))|}{|B_{n,f}(\Delta_1^i)|} - 
 			\log  \frac{|h(\mathcal{I}_i)|}{|\mathcal{I}_i|}
 			\right| \leq 2CM^2\, \lambda^n \ \ \text{and} \ \ \left| \log  \frac{|B_{n,g} (h(\widehat{\Delta}_n^{-i}))|}{|B_{n,f}(\widehat{\Delta}_n^{-i})|} - \log  \frac{|h(\mathcal{J}_i)|}{|\mathcal{J}_i|}
 			\right| \leq 2CM^2\, \lambda^n.
 		\end{equation*}
 \end{claim}
 \begin{proof}[Proof of Claim \ref{claimI}] By definition of $\psi$ and $\phi$, Theorem \ref{thm:Mcontrolledmcm}, Mean Value Theorem and inequality \eqref{inq:estimativacentrofunnel} we obtain:
 	\begin{align*}
 	\left| \log  \frac{|B_{n,g} (h(\Delta_1^i))|}{|B_{n,f}(\Delta_1^i)|} - 
 	\log  \frac{|h(\mathcal{I}_i)|}{|\mathcal{I}_i|}
 	\right| 
 	&= \left| \log  \frac{|(\mathcal{R}^ng)^{i}(1) - (\mathcal{R}^ng)^{i+1}(1)|}{|(\mathcal{R}^nf)^{i}(1) - (\mathcal{R}^nf)^{i+1}(1)|} - 
 	 \log  \frac{|\psi(\mathcal{R}^ng)^{i}(1) - \psi (\mathcal{R}^ng)^{i+1}(1)|}{|\phi(\mathcal{R}^nf)^{i}(1) - \phi(\mathcal{R}^nf)^{i+1}(1)|}
 	\right| \\
 	&= \left| \log  \frac{|(\mathcal{R}^ng)^{i}(1) - (\mathcal{R}^ng)^{i+1}(1)|}{|\psi(\mathcal{R}^ng)^{i}(1) - \psi(\mathcal{R}^ng)^{i+1}(1)|} - 
 	\log  \frac{|(\mathcal{R}^nf)^{i}(1) -  (\mathcal{R}^nf)^{i+1}(1)|}{|\phi(\mathcal{R}^nf)^{i}(1) - \phi(\mathcal{R}^nf)^{i+1}(1)|}
 	\right|\\
 	&= \left| \log \frac{D^2\mathcal{R}^nf(z_n)}{D^2\mathcal{R}^ng(h(z_n))}
 	\right|\\
 	&\leq M\, \left|(D^2\mathcal{R}^ng)(h(z_n))- (D^2\mathcal{R}^nf)(z_n) \right|\\
 	&\leq M\, |h(z_n) -z_n| \, | (D^2\mathcal{R}^nf)(\beta) + (D^2\mathcal{R}^ng)(\alpha) |	\\
 	&\leq 2M^2C \, \lambda^n.
 	% 	&\leq 	\widetilde{C}_2 \, \lambda^n.
 	\end{align*}
Analogously,
	\begin{align*}
	\left| \frac{|B_{n,g} (h(\widehat{\Delta}_n^{-i}))|}{|B_{n,f}(\widehat{\Delta}_n^{-i})|} - 
	\log  \frac{|h(\mathcal{J}_i)|}{|\mathcal{J}_i|}
	\right| 
	%	&= \left| \log  \frac{|(\mathcal{R}^ng)^{-i}(B_{n,g}(h(\mathfrak{c}_m)) - (\mathcal{R}^ng)^{-i-1}(B_{n,g}(h(\mathfrak{c}_m)))|}{|(\mathcal{R}^nf)^{-i}(B_{n,f}(\mathfrak{c}_m)) - (\mathcal{R}^nf)^{-i-1}(B_{n,f}(\mathfrak{c}_m))|} - 
	%	\log  \frac{|\psi(\mathcal{R}^ng)^{-i}(B_{n,g}(h(\mathfrak{c}_m))) - \psi (\mathcal{R}^ng)^{-i-1}(B_{n,g}(h(\mathfrak{c}_m)))|}{|\phi(\mathcal{R}^nf)^{-i}(B_{n,f}(\mathfrak{c}_m)) - \phi(\mathcal{R}^nf)^{-i-1}(B_{n,f}(\mathfrak{c}_m))|}
	%	\right| \\
	%	&= \left| \log  \frac{|(\mathcal{R}^ng)^{-i}(B_{n,g}(\mathfrak{c}_m^g)) - (\mathcal{R}^ng)^{-i-1}(B_{n,g}(\mathfrak{c}_m^g))|}{|\psi(\mathcal{R}^ng)^{-i}(B_{n,g}(\mathfrak{c}_m^g)) - \psi(\mathcal{R}^ng)^{-i-1}(B_{n,g}(\mathfrak{c}_m^g))|} - 
	%	\log  \frac{|(\mathcal{R}^nf)^{-i}(B_{n,f}(\mathfrak{c}_m)) -  (\mathcal{R}^nf)^{-i-1}(B_{n,f}(\mathfrak{c}_m))|}{|\phi(\mathcal{R}^nf)^{-i}(B_{n,f}(\mathfrak{c}_m)) - \phi(\mathcal{R}^nf)^{-i-1}(B_{n,f}(\mathfrak{c}_m))|}
	%	\right|\\
		= \left| \log \frac{D^2\mathcal{R}^nf(z_n)}{D^2\mathcal{R}^ng(h(z_n))}
		\right| \leq 2M^2C \, \lambda^n.
		% 	&\leq 	\widetilde{C}_2 \, \lambda^n.
	\end{align*}
 	\end{proof}

 By Claim \ref{claimI} we just need to prove that there exists a universal constant $\widetilde{C}_1>0$ such that for $i \geq L+P$
  \begin{equation}\label{eq:intervalosnuevosn}
  	\left| 	\log  \frac{|h(\mathcal{I}_i)|}{|\mathcal{I}_i|}
  	\right| \leq \widetilde{C}_1\, \lambda^n \ \ \text{and} \ \	\left| 	\log  \frac{|h(\mathcal{J}_i)|}{|\mathcal{J}_i|}
  	\right| \leq \widetilde{C}_1\, \lambda^n.
  \end{equation}
Using inequality \eqref{eq_funnelsonsecutive} with $s_0=\mathcal{F}_n^{L+P}(\phi(1)) \in (0,1)$, we obtain for $L+P \leq i \leq \max \{i_{r}, j_{r}\}$
\begin{equation}\label{eq:FunnelFmGm}
\frac{h(s_i)- h(s_{i+1})}{s_i-s_{i+1}} = \frac{ \left(i-(L+P)+ (\mathcal{F}_n^{L+P}(\phi(1)))^{-1}  \right)^{2}}{ \left( i-(L+P)+ (\mathcal{G}_n^{L+P}(h(\phi(1))))^{-1} \right)^{2} } \frac{1+B_i}{1+A_i}
\end{equation}
  where $A_i$ and $B_i$ correspond to the constant $\delta_i$ in Lemma \ref{lemaKTtunnel} for $\mathcal{F}_n$ and $\mathcal{G}_n$, respectively. Therefore,  $|A_i| \leq D_2\,s_0$ and $|B_i|\leq D_2h(s_0)$. In fact, by Claim \ref{claimP},  $|A_i|,|B_i| \leq D_2\, \lambda_*^{n/2}$, for $\lambda_*=\max\{\lambda, \lambda^{1/3} \}$.\\
  
  Using that $|\log (a+1/b+1)  \leq |\log (a/b)|$ for $a,b>0$ we have  from equation \eqref{eq:FunnelFmGm}
  \begin{align*}
  	\log \left( \frac{1+A_i}{1+B_i} \ \frac{h(s_i)- h(s_{i+1})}{s_i-s_{i+1}} \right) &= 2 \log \left( \frac{ i-(L+P)+ (\mathcal{F}_n^{L+P}(\phi(1)))^{-1} }{  i-(L+P)+ (\mathcal{G}_n^{L+P}(h(\phi(1))))^{-1} } \right)  \\
  &=2 \log \left(
  		\frac{ 1+(i-(L+P)-1) + (\mathcal{F}_n^{L+P}(\phi(1)))^{-1} }{1+( i-(L+P)-1)+ (\mathcal{G}_n^{L+P}(h(\phi(1))))^{-1}} \right) \\
  &\leq 2 \log \left(
  		\frac{ i-(L+P)-1 + (\mathcal{F}_n^{L+P}(\phi(1)))^{-1} }{ i-(L+P)-1+ (\mathcal{G}_n^{L+P}(h(\phi(1))))^{-1}} \right)  \\
  		&\leq \cdots \leq 2 \log \left( \frac{\mathcal{G}_n^{L+P}(h(\phi(1))}{\mathcal{F}_n^{L+P}(\phi(1))} \right) \\ 
  		&=2 \log\left( \frac{h(s_{L+P}) - h(s_{L+P+1})}{s_{L+P}+ s_{L+P+1}} \, \frac{1+A_{L+P}}{1+B_{L+P}} \right).
  \end{align*}
  Therefore, for $L+P \leq i \leq \max\{i_r,j_r\}$
  \begin{equation*}
  	\left| 	\log  \frac{|h(\mathcal{I}_i)|}{|\mathcal{I}_i|}
  	\right| = \left| \log \frac{|h(s_i)- h(s_{i+1})|}{|s_i-s_{i+1}|} \right| \leq 4(L+1)C_3C\lambda_*^{n/2} + D_2 \lambda_*^{n/2}
  	\leq \widetilde{C}_2\,\lambda_*^{n/2}.
  \end{equation*}
 Analogously, this time for $\mathcal{F}_n^{-1}$ and $\mathcal{G}_n^{-1}$, we have for $\min\{i_{\ell},j_{\ell}\}  \leq i \leq L+P$
  \begin{equation*}
  	\left| 	\log  \frac{|h(\mathcal{I}_i)|}{|\mathcal{I}_i|}
  	\right| \leq \widetilde{C}_2\,\lambda_*^{n/2}.
  \end{equation*}
  Moreover, using inequality \eqref{eq_funnelsonsecutive} for $\mathcal{F}_n^{i}(\phi(B_{n,f}(\mathfrak{c}_n)))$  we obtain for any $\min\{\widehat{i_{\ell}},\widehat{j_{\ell}}\}  \leq i \leq \max \{\widehat{i_{r}}, \widehat{j_{r}}\}$ 
   \begin{equation*}
  	\left| 	\log  \frac{|h(\mathcal{J}_i)|}{|\mathcal{J}_i|}
  	\right| \leq \widetilde{C}_2\,\lambda_*^{n/2}.
  \end{equation*}
  
  Now, for the intervals inside the Tunnel, we will use Lemma \ref{lemaKTtunnel}, with  $s_0 = \mathcal{F}_n^{i_c}(\phi(1))$ with $s_k= \mathcal{F}_n^{i_c+k}(\phi(1))$, for $0 \leq k \leq i_{\ell}-i_c$, and also for $s_0 = \mathcal{G}_n^{j_c}(\psi(1))$ with $s_j= \mathcal{G}_n^{j_c+j}(\psi(1))$, for $0 \leq j \leq j_{\ell}-j_c$   By Equation \eqref{eq_tunnelconsecutive} in  Lemma \ref{lemaKTtunnel} and Lemma  \ref{lemma:estimatesparameterstubular} we have,
  for $i_c \leq i \leq \min \{ i_{\ell}, j_{\ell}\}$
 {\small
  \begin{align*}
  	\left| \log  \frac{|h(\mathcal{I}_{i})|}{|\mathcal{I}_{i}|} \right| 
  	&\leq 
  	\left| \log\frac{|h(\mathcal{I}_{i})|}{|\mathcal{I}_{i}|} - 	\log
  	\frac{|\mathcal{G}_n^{j_c+i-i_c+1}(\psi(1))-\mathcal{G}_n^{j_c+i-i_c}(\psi(1)))| }{|\mathcal{F}_n^{i+1}(\phi(1))- \mathcal{F}_n^{i}(\phi(1))|}   	\right| 
  	+ \left|\log
  	\frac{|\mathcal{G}_n^{j_c+i-i_c+1}(\psi(1))-\mathcal{G}_n^{j_c+i-i_c}(\psi(1)))| }{|\mathcal{F}_n^{i+1}(\phi(1))- \mathcal{F}_n^{i}(\phi(1))|} \right|\\
  	&\leq \left| \sum_{m=i}^{j_c-i_c+i} \log \frac{ |\mathcal{G}_n(h(s_m))-\mathcal{G}_n(h(s_{m+1}))| }{|h(s_m)-h(s_{m+1})|}
  	\right| + \left| \log \frac{|h(\varepsilon_n)(1+\delta_{i+j_c-i_c})(\cos \sqrt{\varepsilon_n}(i-i_c))^2|}{|\varepsilon_n(1+\delta_i) (\cos \sqrt{h(\varepsilon_n)}(i-i_c))^2|} \right|
  	\\
  	&\leq |j_c-i_c| \varepsilon_n^{\frac{1}{2+\alpha}} +  
  	\left| \log \frac{h(\varepsilon_n)(1+\delta_{i+jc-i_c})}{\varepsilon_n(1+\delta_i)} \right| + \left|2 \log \frac{\cos \sqrt{\varepsilon_n}(i-i_c)}{\cos \sqrt{h(\varepsilon_n)}(i-i_c)} \right|\\
  	&\leq D_5\varepsilon_n^{\frac{\alpha(\alpha+1)}{2(2+\alpha)}} +  D_5\varepsilon_n^{\alpha/2} + 2D_4\,\varepsilon_n^{\frac{\alpha(\alpha+1)}{2(2+\alpha)}} + \left|2 \log \frac{\cos \sqrt{\varepsilon_n}(i_{\ell})}{\cos \sqrt{h(\varepsilon_n)}(i_{\ell})} \right| \\
  	&\leq D_6\,\lambda_*^{\frac{n\alpha(\alpha +1)}{4(2+\alpha)}}  +\widetilde{C}_2\,\lambda_*^{n/2}\\
  	&\leq \widetilde{C}_4\,\lambda_*^{n/2},
  \end{align*}
}
where $D_6=(D_5, M, \alpha)>0$.  Analogously we obtain a similar inequality for $\max\{i_r,j_r\} \leq i \leq i_{c}$. Therefore, we have for  $ \max\{i_r,j_r\} \leq i \leq \min \{ i_{\ell}, j_{\ell}\}$   
  \begin{align*}
  	\left| 	\log  \frac{|h(\mathcal{I}_i)|}{|\mathcal{I}_i|}
  	\right| 
  	\leq \widetilde{C}_4\, \lambda_*^{n/2}.
  \end{align*}
  where $\widetilde{C}_4= \widetilde{C}_4(L,M, D_4,D_5,C, C_3)>0$.

Analogously, we obtain the same estimates for $	\left| 	\log  \frac{|h(\mathcal{J}_i)|}{|\mathcal{J}_i|}
\right|$ where  $ \min\{ \widehat{i_{\ell}}, \widehat{j_{\ell}} \} \leq i\leq \max \{ \widehat{i_r}, \widehat{j_r} \}$. Note also that by the estimatives above and Remark \ref{rmk:comparabilitytwopartions} we have the proposition also for  $I=\Delta_R^{n+1}$.
Finally, we obtain the result putting all previous inequalities together and choosing  $\lambda_3=\max\{ \lambda_2 ^{1/2},\lambda^{1/2}, \lambda_*^{n/2} \}$ and $C_4>0$ depending on $L,M,D_2,C,C_3, D_4$ and $D_5$.
\end{proof}

The next result is an inductive step between $J_m$ and $J_{m+1}$.

\begin{lemma}\label{lema_pasocuatro} There exists $C_5=C_5(M,D_1,C_3) > 1$ and $\lambda_5 \in (0,1)$ such that for all $(1-b_1)n \leq m \leq n-1$, $I \in \widehat{\mathcal{P}}_{n+1}$ and $I \subset J_{m} $
	\begin{align*}
		\left| \log \frac{|B_{m,g} (h(I))|}{|B_{m,f} (I)|} \right| \leq C_5 \left(\max_{I  \in \widehat{\mathcal{P}}_{n},
			I \subset J_{m}}  \left| \log \frac{|B_{m,g} (h(I))|}{|B_{m,f} (I)|} \right| +
		\max_{I  \in \widehat{\mathcal{P}}_{n+1},
			I \subset J_{m+1}}  \left| \log \frac{|B_{m+1,g} (h(I))|}{|B_{m+1,f} (I)|} \right|
		  + \lambda_2^{m/4} \right)
	\end{align*}
\end{lemma}
\begin{proof}[Proof of Lemma \ref{lema_pasocuatro}]
	If $I \in \widehat{\mathcal{P}}_{n+1}$ and $I \subset J_{m+1}$ then
	\begin{equation*}
		\left| \log \frac{|B_{m,g} (h(I))|}{|B_{m,f} (I)|} \right| 
		\, = \, \left| \log \frac{|B_{m+1,g} (h(I))|}{|B_{m+1,f} (I)|} + \log \frac{|\mathcal{R}^mg(0)|}{|\mathcal{R}^mf(0)|} \right| \, \leq \, \left| \log \frac{|B_{m+1,g} (h(I))|}{|B_{m+1,f} (I)|} \right|  +  C\, \lambda^{m},
	\end{equation*} 
and hence the lemma follows for these intervals. We just need to prove the lemma for $I \in \widehat{\mathcal{P}}_{n+1}$ and $I \subset J_{m} \setminus J_{m+1}$. Without loss of generality let us assume also that $I \subset [\mathfrak{c}_m, f^{q_m}(c_0)]$. By simplicity, let us call $\widehat{\Delta} \in \widehat{\mathcal{P}}_{m+1}$ the interval with endpoints given by  $f^{q_m}(\mathfrak{c}_m)$ and $\mathfrak{c}_m$. 
\begin{claim}\label{claim:mm+1Delta}
	For any $I \in \widehat{\mathcal{P}}_{n+1}$ with $I \subset \widehat{\Delta}$ we have
	\begin{equation}\label{ineq_mm+1claim}
			\left| \log \frac{|B_{m,g} (h(I))|}{|B_{m,f} (I)|} \right| \leq
		\left| \log \frac{|B_{m+1,g} (h(I))|}{|B_{m+1,f} (I)|} \right|
		+ C\lambda^{m}.
	\end{equation}
\end{claim}
\begin{proof}[Proof of Claim \ref{claim:mm+1Delta}]
Let $j < q_{m+2}$ be such that $f^{j}(\widehat{\Delta})=I_{m+1}(c_1)$, that is the fundamental interval with endpoints $f^{q_{m+1}}(c_1)$ and $c_1$. If $I \in \widehat{\mathcal{P}}_{n+1}$ with $I \subset I_{m+1}(c_1)$ then
    \begin{equation}\label{ineq:m1m+11}
    	\left| \log \frac{|B_{m,g,1} (h(I))|}{|B_{m,f,1} (I)|} \right| =  \left| \log \frac{|B_{m+1,g,1} (h(I))| |\mathcal{R}^mg(0)|}{|B_{m+1,f,1} (I)||\mathcal{R}^mf(0)|} \right|
    	 %\max_{\widehat{I}  \in \widehat{\mathcal{P}}_{n},
    	%	\widehat{I} \subset J_{m}(c_1)}  \left| \log \frac{|B_{m,g,1} (h(\widehat{I}))|}{|B_{m,f,1} (\widehat{I})|} \right| +
    	%\max_{\widehat{I} \in \widehat{\mathcal{P}}_{n+1},
    		%\widehat{I} \subset J_{m+1}(c_1)}  
    	\leq 
    	\left| \log \frac{|B_{m+1,g,1} (h(I))|}{|B_{m+1,f,1} (I)|} \right|
    	+ C\lambda^{m},
    \end{equation}
    where $B_{m+1,g,1}(t)=(-1)^n |t-h(c_1)|/|h(I_{m+1}(c_1))|$. 
    Let $\widehat{I} \in \widehat{\mathcal{P}}_{n+1}$, $\widehat{I} =[v,w] \subset \widehat{\Delta}$. By Lemma \ref{lemma_decompositiontwobridge}, for any $\varepsilon>0$, we have that 
\begin{equation*}
	\sup_{x,y \in h(\widehat{\Delta})} \frac{|Dg^j(x)|}{|Dg^j(y)|} \leq 1+ \varepsilon,
\end{equation*}
which implies
    \begin{equation*}
    	\frac{|g^{j}(g^{q_{m+1}}(h(\mathfrak{c}_1))-g^{j}(h(\mathfrak{c}_1))|}{|g^{q_{m+1}}(h(\mathfrak{c}_1))-h(\mathfrak{c}_1)|}
    		\leq (1+ \varepsilon) \frac{|g^{j}(h(w))-g^{j}(h(v))|}{|h(v) -h(w)|}. 
    \end{equation*}
These two estimates imply that:    
    \begin{align*}
    	\frac{|B_{m,g} (h(\widehat{I}))|}{|B_{m,f} (\widehat{I})|} %= \frac{\frac{|h(w)-h(v)|}{|h(I_{m}(c_0))|}}{\frac{|w-v|}{|I_{m}(c_0)|}}  
    	\leq   \frac{\frac{|h(w)-h(v)|}{|g^{q_{m+1}}(h(\mathfrak{c}_1))-h(\mathfrak{c}_1))|}}{\frac{|w-v|}{|f^{q_{m+1}}(\mathfrak{c}_1) -\mathfrak{c}_1|}}  \leq 
    	(1+\varepsilon)^2\, 
    	\frac{\frac{|g^{j}(h(w))-g^{j}(h(v))|}{|g^{q_{m+1}}(h(c_1))-h(c_1)|}}{\frac{|f^{j}(w)-f^{j}(v)|}{|f^{q_{m+1}}(c_1) -c_1|}} \leq (1+\varepsilon)^2 \, \frac{|B_{m,g,1} (g^{j}(h(\widehat{I}))|}{|B_{m,f,1} (f^{j}(\widehat{I}))|}.
    \end{align*}
 In a similar way we obtain
 \begin{align*}
	\left| \log \frac{|B_{m+1,g,1} (g^{j}(h(\widehat{I}))|}{|B_{m+1,f,1} (f^{j}(\widehat{I}))|} \right|
	\leq \left| \log \frac{|B_{m+1,g} (h(\widehat{I}))|}{|B_{m+1,f} (\widehat{I})|} \right|.
\end{align*}
 From \eqref{ineq:m1m+11} and $f^j(\widehat{I}) \subset I_{m+1}(c_1)$ we  obtain
\begin{equation*}
	\left| \log \frac{|B_{m,g} (h(\widehat{I}))|}{|B_{m,f} (\widehat{I})|} \right|  \leq  \left| \log \frac{|B_{m,g,1} (g^{j}(h(\widehat{I})))|}{|B_{m,f,1} (f^{j}(\widehat{I}))|} \right| \leq  
	\left| \log \frac{|B_{m+1,g,1} (g^{j}h(\widehat{I}))|}{|B_{m+1,f,1} (f^{j}(\widehat{I}))|} \right|
	+ C\lambda^{m} \leq \left| \log \frac{|B_{m+1,g} (h(\widehat{I}))|}{|B_{m+1,f} (\widehat{I})|} \right| + C\lambda^{m}.
\end{equation*}
    Hence,  any interval  $I \in \widehat{\mathcal{P}}_{n+1}$ contained in $\widehat{\Delta}$ satisfies \eqref{ineq_mm+1claim}.
    \end{proof}
     On one hand, by Claim \ref{claim:mm+1Delta} and Proposition \ref{Pro_pasodos},   any interval  $I \in \widehat{\mathcal{P}}_{n+1}$ is contained in $f^{-q_{m+1}}(\widehat{\Delta})$ satisfies inequality \eqref{ineq_mm+1claim}. Let  $\Delta \in \mathcal{P}_{m+1}$ with $\Delta \subset\widehat{\Delta} \cup f^{-q_{m+1}}(\widehat{\Delta})$, then any interval  $I \in \widehat{\mathcal{P}}_{n+1}$ contained in $\Delta$ also satisfies inequality \eqref{ineq_mm+1claim}.  On the other hand, by Corollary \ref{coro:Pro_pasodos}, any interval $I \in \widehat{\mathcal{P}}_{n+1}$ contained in $f^{q_m}(I_{m+1})$ satisfies inequality \eqref{ineq_mm+1claim}. This implies that  any $I \in \widehat{\mathcal{P}}_{n+1}$ contained in $f^{q_{m+1}+q_m}(I_{m+1}) \cup f^{q_m}(I_{m+1})$ satisfies inequality \eqref{ineq_mm+1claim}.  Now let  $P=\lceil{\lambda^{-m/4}}\rceil$, we iterate $P+1$ times forward and backward the  map  $f^{q_{m+1}}$ from the intervals $f^{q_m}(I_{m+1})$ and $\widehat{\Delta}$, respectively.
    By Proposition \ref{Pro_pasodos}, for any  $I \in \widehat{\mathcal{P}}_{n+1}$ contained in $\Delta$ and for any $1 \leq j \leq P+1$ we have
    	\begin{equation}\label{ineq:hatDelta}
    		\left| \log \frac{|B_{m,g} (g^{-jq_{m+1}}(h(I)))|}{|B_{m,f} (f^{-jq_{m+1}}(I))|} \right|  \\
    		\leq
    		%\left| \log \frac{|B_{m,g} g^{-jq_{m+1}}(h(I))|}{|B_{m,f}f^{-jq_{m+1}}(I)|} -  \log \frac{|B_{m,g} g^{-(P+1)q_{m+1}}(h(I))|}{|B_{m,f} f^{-(P+1)q_{m+1}}(I)|} \right| + \left| \log \frac{|B_{m,g} g^{-(P+1)q_{m+1}}(h(I))|}{|B_{m,f} f^{-(P+1)q_{m+1}}(I))|} \right|\\
    		%\hspace{0.4cm}
    		%\leq  \dots \leq
    		%C_3(P+1-j) \lambda_3^{n/2} + \left| \log \frac{|B_{m,g} (g^{-(P+1)q_{m+1}}(h(I)))|}{|B_{m,f} (f^{-(P+1)q_{m+1}}(I))|} \right| \hspace{4.9cm}  \\    		
    		%\leq C_3 \lambda^{m/4} + \left| \log \frac{|B_{m,g} (g^{(P+1)q_{m+1}}(h(I)))|}{|B_{m,f} (f^{(P+1)q_{m+1}}(I))|} \right|. \hspace{7.7cm} 
    		2C_3 \lambda^{m/4} +  \left| \log \frac{|B_{m+1,g} (h(I))|}{|B_{m+1,f} (I)|} \right| .
    \end{equation} 
  Moreover,  for any $I \in \widehat{\mathcal{P}}_{n+1}$ contained in $f^{q_m}(I_{m+1})$, for any $1 \leq j \leq P+1$ we have
	\begin{equation}\label{ineq:Delta}
		\left| \log \frac{|B_{m,g} (g^{jq_{m+1}}(h(I)))|}{|B_{m,f} (f^{jq_{m+1}}(I))|} \right| \leq
		%\left| \log \frac{|B_{m,g} g^{jq_{m+1}}(h(I))|}{|B_{m,f}f^{jq_{m+1}}(I)|} -  \log \frac{|B_{m,g} g^{(P+1)q_{m+1}}(h(I))|}{|B_{m,f} f^{(P+1)q_{m+1}}(I)|} \right| + \left| \log \frac{|B_{m,g} g^{(P+1)q_{m+1}}(h(I))|}{|B_{m,f} f^{(P+1)q_{m+1}}(I))|} \right|\\
		%\hspace{0.4cm}
		%\leq  \dots \leq
		%C_3(P+1-j) \lambda_3^{n/2} + \left| \log \frac{|B_{m,g} (g^{(P+1)q_{m+1}}(h(I)))|}{|B_{m,f} (f^{(P+1)q_{m+1}}(I))|} \right| \\
		%\leq C_3 \lambda_3^{m/4} + \left| \log \frac{|B_{m,g} (g^{(P+1)q_{m+1}}(h(I)))|}{|B_{m,f} (f^{(P+1)q_{m+1}}(I))|} \right| \hspace{2cm}\\
		2C_3 \lambda^{m/4} +  \left| \log \frac{|B_{m+1,g} (h(I))|}{|B_{m+1,f} (I)|} \right| .
	\end{equation} 
    If $r(m) \leq \max\{2(P+1),2L\}$ then we have proved the lemma. If $r(m) > \max\{2(P+1),2L\}$, then the tubular $M_{m,2L}^f$ is not empty and its center $z_m$ is defined. We need to prove the lemma for  $I_i \in \widehat{\mathcal{P}}_{n+1}$ which either $I_i \subseteq f^{-iq_{m+1}}(\widehat{\Delta})$ or $I_i\subseteq f^{q_m+iq_{m+1}}(I_{m+1})$ for $P+1\leq i \leq r(m)- P-1$.
    Let $J_i$ be the unique interval in $\widehat{\mathcal{P}}_{n}$ with $I_i \subseteq J_i$. 
    Then
    \begin{multline}\label{ineq:nextlevelJiIi}
    	\left| \log \frac{|B_{m,g}(h(I_i)|}{|B_{m,f}(I_i)|} \right| \leq 
    	\left| \log \frac{|B_{m,g}(h(J_i))|}{|B_{m,f}(J_i)|} - \log \frac{|B_{m,g}(h(I_i))|}{|B_{m,f}(I_i)|} \right| + 	\left| \log \frac{|B_{m,g}(h(J_i))|}{|B_{m,f}(J_i)|} \right|\\
    	\leq 
    	\left| \log \frac{|B_{m,g}(h(J_i))|}{|B_{m,f}(J_i)|} - \log \frac{|B_{m,g}(h(I_i))|}{|B_{m,f}(I_i)|} \right| + 	  \max_{I \in \widehat{\mathcal{P}}_{n}, I \subset J_m} \left| \log \frac{|B_{m,g}(h(I))|}{|B_{m,f}(I)|} \right|. 
    	\hspace{1.3cm} 
    \end{multline}
    On the other hand, 
     \begin{multline}\label{ineq:JiIi}
    	\left| \log \frac{|B_{m,g}(h(J_i))|}{|B_{m,f}(J_i)|} - \frac{|B_{m,g}(h(I_i))|}{|B_{m,f}(I_i)|} \right|   \\
     \hspace{0.5cm}	\leq 	\left| \log \frac{|B_{m,g}(h(J_i))|}{|B_{m,f}(J_i)|} - \frac{|B_{m,g}(h(I_i))|}{|B_{m,f}(I_i)|} +  \log \frac{|B_{m,g}(g^{ q_{m+1}}(h(J_i))|}{|B_{m,f}f^{q_{m+1}}(J_i)|} - \log \frac{|B_{m,g}(g^{q_{m+1}}(h(I_i))|}{|B_{m,f}f^{q_{m+1}}(I_i)|}   \right|  \\ 
    	+    	\left| \log \frac{|B_{m,g}g^{q_{m+1}}(h(J_i))|}{|B_{m,f}f^{q_{m+1}}(J_i)|} - \log \frac{|B_{m,g}g^{q_{m+1}}(h(I_i))|}{|B_{m,f}f^{q_{m+1}}(I_i)|} \right|.
    \end{multline}
  
    By Mean Value Theorem there exists $\widetilde{\theta} \in B_{m,f}(J_i)$ and $\theta \in B_{m,f}(I_i)$   such that
    \begin{align*}
    	\left| \log \frac{|B_{m,f} (f^{q_{m+1}}(J_i))|}{|B_{m,f} (f^{q_{m+1}}(I_i))|} - \log \frac{|B_{m,f}(J_i)|}{|B_{m,f}(I_i)|} \right|
&= 	\left| \log \frac{|B_{m,f} (f^{q_{m+1}}(J_i))|}{|B_{m,f}(J_i)|} - \log \frac{|B_{m,f} (f^{q_{m+1}}(I_i))|}{|B_{m,f}(I_i)|} \right| \\
&\leq |\log D(\mathcal{R}^mf)(\theta)- \log D (\mathcal{R}^mf)(\widetilde{\theta}) | .
   \end{align*}
 Note that there are no critical points of $\mathcal{R}^mf$ between $\theta$ and $\widetilde{\theta}$ since, in the worse case scenario, $I_i$ and $J_i$ have  $B_{m,f}(\mathfrak{c}_m)$ as an endpoint in common and then $B_{m,f}(\mathfrak{c}_m)$ cannot be between $\theta$ and $\widetilde{\theta}$\footnote{Note that if we were using the classical dynamical partitions it could happen that $B_{m,f}(\mathfrak{c}_m)$ is between $\theta$ and $\widetilde{\theta}$.}. Using the Mean Value Theorem again there exists  $\widehat{\theta}$, between $\theta$ and $\widetilde{\theta}$,  such that
 {\small
 \begin{align*}
 	|\log D(\mathcal{R}^mf)(\theta)- \log D (\mathcal{R}^mf)(\widetilde{\theta}) | =  D\left(\log D(\mathcal{R}^mf) \right) (\widehat{\theta})| \, |\theta- \widetilde{\theta}| \leq \frac{|D^2(\mathcal{R}^mf)(\widehat{\theta})|}{|D(\mathcal{R}^mf)(\widehat{\theta})|} \, |B_{m,f}(J_i)|
 	\leq M^2 \, |B_{m,f}(J_i)|.
\end{align*}
}
If $J_i \subset [f^{q_m+(r(m)-1)q_{m+1}}(c_0), f^{q_m}(c_0)]$, then  $|B_{m,f}(J_i)| = \frac{M}{2}|\mathcal{F}_m^{i}(\phi^{-1}(1)) - \mathcal{F}_m^{i-1}(\phi^{-1}(1))|$ and hence  
\begin{equation*}
	\left| \log \frac{|B_{m,f} (f^{q_{m+1}}(J_i))|}{|B_{m,f} (f^{q_{m+1}}(I_i))|} - \log \frac{|B_{m,f}(J_i)|}{|B_{m,f}(I_i)|} \right|
	 \leq   \frac{M^3}{2} \ |\mathcal{F}_m^{i}(\phi^{-1}(1)) - \mathcal{F}_m^{i-1}(\phi^{-1}(1))|
	\end{equation*}
and analogously
\begin{equation*}
\left| \log \frac{|B_{m,g} (g^{q_{m+1}}h((J_i))|}{|B_{m,g}( g^{q_{m+1}}h(I_i))|} - \log \frac{|B_{m,g}h(J_i)|}{|B_{m,g}h(I_i)|} \right| \leq \frac{M^3}{2} \ |\mathcal{G}_m^{i}(\psi^{-1}(1)) - \mathcal{G}_m^{i-1}(\psi^{-1}(1))|.
\end{equation*}
Consequently inequality \eqref{ineq:JiIi} becomes
\begin{multline}\label{ineq:JiIiinduction}
	\left| \log \frac{|B_{m,g}(h(J_i))|}{|B_{m,f}(J_i)|} - \log \frac{|B_{m,g}(h(I_i))|}{|B_{m,f}(I_i)|} \right| \leq \frac{M^3}{2}\, \left( |\mathcal{G}_m^{i}(\psi(1)) - \mathcal{G}_m^{i-1}(\psi(1))| + |\mathcal{F}_m^{i}(\phi(1)) - \mathcal{F}_m^{i-1}(\phi(1))|\right) \\
	+ \hspace{0.2cm}	\left| \log \frac{|B_{m,g}g^{q_{m+1}}(h(J_i))|}{|B_{m,f}f^{q_{m+1}}(J_i)|} - \log \frac{|B_{m,g}g^{q_{m+1}}(h(I_i))|}{|B_{m,f}f^{q_{m+1}}(I_i)|} \right|.
\end{multline}
Applying inequality \eqref{ineq:JiIiinduction} inductively, we obtain for  $P+1\leq i \leq r(m)- P-1$,
\begin{multline}\label{ine:sumJiIi}
	\left| \log \frac{|B_{m,g}(h(J_i))|}{|B_{m,f}(J_i)|} - \log \frac{|B_{m,g}(h(I_i))|}{|B_{m,f}(I_i)|} \right| \\
	\hspace{-1.3cm} \leq \frac{M^3}{2}\, \left( \sum_{i=P+1}^{2r(m)-P-1} |\mathcal{G}_m^i(\psi(1)) - \mathcal{G}_m^{i+1}(\psi(1))| + \sum_{i=P+1}^{2r(m)-P-1}  |\mathcal{F}_m^i(\phi(1)) - \mathcal{F}_m^{i+1}(\phi(1))|\right) \\
	 \hspace{2cm}	 + \left| \log \frac{|B_{m,g}g^{(2r(m)-P-1)q_{m+1}}(h(J_i))|}{|B_{m,f}f^{(2r(m)-P-1)q_{m+1}}(J_i)|} - \log \frac{|B_{m,g}g^{(2r(m)-P-1)q_{m+1}}(h(I_i))|}{|B_{m,f}f^{(2r(m)-P-1)q_{m+1}}(I_i)|} \right| \\
	\hspace{-1.4cm}  \leq \frac{M^3}{2} \left(  | \mathcal{G}_m^{P-1}(\psi(1)) - \mathcal{G}_m^{(2r(m)-P-1)}(\psi(1)) | +  | \mathcal{F}_m^{P-1}(\phi(1)) - \mathcal{F}_m^{(2r(m)-P-1)}(\phi(1)) |   \right) \\
	\hspace{2cm} + \left| \log \frac{|B_{m,g}g^{(2r(m)-P-2)q_{m+1}}(h(J_i))|}{|B_{m,f}f^{(2r(m)-P-1)q_{m+1}}(J_i)|} \right| + \left| \log \frac{|B_{m,g}g^{(2r(m)-P-1)q_{m+1}}(h(I_i))|}{|B_{m,f}f^{(2r(m)-P-1)q_{m+1}}(I_i)|} \right|  \\
	\leq \frac{M^5(D_1+1)}{4} \lambda^{m/2} \ + \max_{I \in \widehat{\mathcal{P}}_n, I \subset J_m} \left| \log  \frac{|B_{m,g}(h(I)|}{|B_{m,f}(I)|} \right| + \left| \log \frac{|B_{m,g}g^{(2r(m)-P-1)q_{m+1}}(h(I_i))|}{|B_{m,f}f^{(2r(m)-P-1)q_{m+1}}(I_i)|} \right|. \hspace{0.5cm} 
\end{multline}
%+ C_3\lambda_3^{m/4} +   \left| \log \frac{|B_{m,g} (g^{(P+1)q_{m+1}}h(I))|}{|B_{m,f}f^{(P+1)q_{m+1}}(I)|} \right|.
The last inequality follows from Lemma \ref{lemaKTfunnel}. \\
Now, let $L, R$ be the two adjacent intervals in $\widehat{\mathcal{P}}_{n+1}$ with $I_i \subset L \cup R$. These three intervals are all  comparable (see Remark \ref{rmk:comparabilitytwopartions} ) and  $R \subset f^{-(P+1)q_{m+1}}(\widehat{\Delta})$ satisfies inequality \eqref{ineq:Delta}, then we have 
\begin{equation}\label{ineq:IiwithI}
	 \left| \log \frac{|B_{m,g}g^{(2r(m)-P-1)q_{m+1}}(h(I_i))|}{|B_{m,f}f^{(2r(m)-P-1)q_{m+1}}(I_i)|} \right|   \leq 
	 \left| \log \frac{|B_{m,g}(h(R))|}{|B_{m,f}(R)|} \right| 
	 <   C_3 \lambda^{m/4} +
	  \left| \log \frac{|B_{m+1,g} (h(R))|}{|B_{m+1,f} (R)|} \right|.
\end{equation}
Finally, from \eqref{ine:sumJiIi} and \eqref{ineq:IiwithI} we have that inequality \ref{ineq:nextlevelJiIi} becomes
\begin{multline}\label{ine:finalJiIi}
	\left| \log \frac{|B_{m,g}(h(I_i))|}{|B_{m,f}(I_i)|}  \right| \leq  C_5 \left(\max_{I \in \widehat{\mathcal{P}}_n, I \subset J_m} \left| \log  \frac{|B_{m,g}(h(I)|}{|B_{m,f}(I)|} \right|
	+  \max_{I \in \widehat{\mathcal{P}}_{n+1}, I \subset J_{m+1}} \left| \log \frac{|B_{m+1,g} (h(I))|}{|B_{m+1,f}(I)|} \right| + \lambda_2^{m/4} \right),
\end{multline}
for $C_5= 2\max\{ \frac{M^{5}(D_1+1)}{4}, C_3\}>1$.\\
Analogously, if  $J_i \subset [f^{q_{m+1}}(\mathfrak{c}_m), f^{(-r(m)+1)q_{m+1}}(\mathfrak{c}_m)]$, we have the same estimate given by \eqref{ine:finalJiIi}.
	\end{proof}

%%%%%%%%%%%%%%%%%%%%%%%%%%%%%%%%%%%%%%%%%%%%

Next we extend the result in Proposition \ref{Pro_pasotres} for intervals in $\widehat{\mathcal{P}}_{n+1}$   contained in $J_m$.

   \begin{proposition}\label{Pro_pasocuatro} Let $C_6>1$ be as in Proposition \ref{Pro_pasotres}. For each $\lambda_5 \in (\lambda_4, 1)$ , where $\lambda_4 = \max\{\lambda_3, \lambda, \lambda_2^{1/4}\}\in (0,1)$, there exist $b_2=b_2(\lambda_5) \in (0,b_1)$  such that for all $I \in \widehat{\mathcal{P}}_{n+1}$, with $I \subset J_{n - \lceil b_2 n \rceil}$,   we have
	\begin{equation*}
		\left| \log \frac{|h(I)|}{|I|}  \right| \leq C_6\, \lambda_5^{n}.
	\end{equation*}
\end{proposition}

\begin{proof}[Proof of Proposition \ref{Pro_pasocuatro}]
	The proof is an induction using Proposition \ref{Pro_pasotres} as first step and Lemma \ref{lema_pasocuatro} in every next step.  By Lemma \ref{lema_pasocuatro} and Proposition \ref{Pro_pasotres} we have that for $I\in \widehat{\mathcal{P}}_{n+1}$ with $I \subset J_{n-1}$
		\begin{align*}
		\left| \log \frac{|B_{n-1,g} (h(I))|}{|B_{n-1,f} (I)|} \right| %&\leq \widehat{C} \left(\max_{I  \in \widehat{\mathcal{P}}_{n},
		%	I \subset J_{n-1}}  \left| \log \frac{|B_{n-1,g} (h(I))|}{|B_{n-1,f} (I)|} \right| +
		%\max_{I  \in \widehat{\mathcal{P}}_{n+1},
		%	I \subset J_{n}}  \left| \log \frac{|B_{n,g} (h(I))|}{|B_{n,f} (I)|} \right|
		%+ \lambda_1^{n-1/4} \right)\\
		&\leq C_5 \left(\max_{I  \in \widehat{\mathcal{P}}_{n},
			I \subset J_{n-1}}  \left| \log \frac{|h(I)| |I_{n-1}|}{|I||h(I_{n-1})|} \right| +
		\max_{I  \in \widehat{\mathcal{P}}_{n+1},
			I \subset J_{n}}  \left| \log \frac{|h(I)||I_n|}{|I||h(I_n)|} \right|
		+ \lambda_2^{n-1/4} \right)
		\\
		&\leq  C_5 \left( C_4\, \lambda_3^{n-1} + C\, \lambda^{n-1} +
		 C_4\, \lambda_3^{n} + C\, \lambda^{n} + \lambda_2^{n-1/4} \right) \\
		 &\leq 5C_5 C_6\, \lambda_4^{n-1},
	\end{align*}
where $C_6=\max \{C_4,C,1\}$ and recall that $\lambda_4 = \max\{\lambda_3, \lambda, \lambda_2^{1/4}\}$.\\
Now, using inequality above we have that for $I\in \widehat{\mathcal{P}}_{n+1}$ with $I \subset J_{n-2}$
\begin{align*}
	\left| \log \frac{|B_{n-2,g} (h(I))|}{|B_{n-2,f} (I)|} \right| &\leq C_5  \left(\max_{I  \in \widehat{\mathcal{P}}_{n},
		I \subset J_{n-2}}  \left| \log \frac{|B_{n-2,g} (h(I))|}{|B_{n-2,f} (I)|} \right| +
	\max_{I  \in \widehat{\mathcal{P}}_{n+1},
		I \subset J_{n-1}}  \left| \log \frac{|B_{n-1,g} (h(I))|}{|B_{n-1,f} (I)|} \right|
	+ \lambda_1^{n-1/4} \right)\\
	&\leq  C_5 (5C_5C_6 \, \lambda_4^{n-2} + 5C_5C_6 \lambda_4^{n-1} +  \lambda_2^{n-2/4})\\
	&\leq (5C_5)^{2}C_6 \, \lambda_4^{n-2}.
 \end{align*}
Inductively we get that for $I\in \widehat{\mathcal{P}}_{n+1}$ with $I \subset J_{n-l}$, for $k \geq 0$, we have
\begin{align*}
	\left| \log \frac{|B_{n-l,g} (h(I))|}{|B_{n-l,f} (I)|} \right| 
	\leq  C_6 (5 C_5)^k\,  \lambda_4^{n-k}.
\end{align*}
In other words,  for all $I\in \widehat{\mathcal{P}}_{n+1}$ with  $I \subset J_{n-k}$, for $k \geq 0$,
\begin{align*}
	\left| \log \frac{|h(I)|}{|I|} \right| 
	%\leq  (3C)^l\, \tilde{C} \, \lambda_2^{n-l} + C\, \lambda^{n-l} \leq  (4C)^l\, \tilde{C} \, \lambda_2^{n-l}
	\leq C_6(5C_5)^k\, \lambda_4^{n-k}.
\end{align*}
Given  $\lambda_5 \in (\lambda_4,1)$, we choose  $b_2 =b_2(\lambda_5) \in (0,b_1)$ such that $ (5C_5)^{b_2} \, \lambda_4^{1-b_2}\leq \lambda_5$.
\end{proof}

Let us finish this section saying that we can obtain the same results in sections $6$ and $7$ replacing the critical point $c_0$ by $c_1$ and $J_m(c_0)$ by $J_m(c_1)$.
Moreover, let us mention that it is also possible to get the results in sections $6$ and $7$ for the sequence of classical dynamical partition $\{\mathcal{P}_n\}$, excepts Lemma \ref{lema_pasocuatro}. 

%\begin{equation*}
%	\widehat{E}_{n+1} \cap (I_m \setminus I_{m+2}) \ \subseteq \bigcup_{v \in \widehat{E}_{n+1} \cap I_{m+1} \setminus \{ f^{q_{m+1}}(c_0)\}} \{f^{q_{m}+jq_{m+1}}(v)\}_{0 \leq j \leq a_{m+1}}
%\end{equation*} 
%%%%%%%%%%%%%%%%%%%%%%%%%%%%%%%%%%%%%%%%%%%%%%%%%%%%%%%%%%%%%%%%
	
	\section{Proof of Main Theorem}\label{sec:proofmaintheorem}
	By Proposition \ref{Pro_pasocuatro}, there exists $C_6>0$, $\lambda_5 \in (0,1)$ and $b_2 \in (0,1)$ such that  for any $I,I' \in \widehat{\mathcal{P}}_{n+1}$ contained in $J_{n - \lceil b_2 n \rceil}(c_i)$, for $i=0$ or $i=1$,  we have that
		\begin{equation}\label{eq:dadaporpro_pasodos}
			\left| \log \frac{|h(I)|}{|I|} - \log \frac{|h(I')|}{|I'|}    \right| \leq C_6\, \lambda_5^{n}
		\end{equation}
	which either are adjacent or are both contained in the same interval of $\widehat{\mathcal{P}}_{n}$. So, to prove our Main Theorem, we just need to obtain  inequality \eqref{eq:dadaporpro_pasodos}  for $I,I' \in\widehat{\mathcal{P}}_{n+1}$, inside the complement of $J_{n - \lceil b_2 n \rceil}(c_0) \cup J_{n - \lceil v_2 n \rceil}(c_1)$, both being adjacent intervals or  contained in the same interval of $\widehat{\mathcal{P}}_{n}$. 
	
	Let  $I,I' \in\widehat{\mathcal{P}}_{n+1}$ such that $I \cup I'$ is not contained in  $J_{n - \lceil b_2 n \rceil}(c_0) \cup J_{n - \lceil b_2 n \rceil}(c_1)$.
	\begin{caseof}
		\case{$I$ and $I'$ are adjacent and also that they are contained in the same interval of $\widehat{\mathcal{P}}_{n}$.} {In this case, $I$ and $I'$ are also contained in the same interval of $\widehat{\mathcal{P}}_{n - \lceil b_2 n \rceil}$, let $\Delta \in \widehat{\mathcal{P}}_{n - \lceil b_2 n \rceil}$ be such interval. Let $\Delta^{*}$ be the union of $\Delta$ with its left and right neighbors in the partition $\widehat{\mathcal{P}}_{n - \lceil b_2 n \rceil}$. Let $j<q_{n - \lceil b_2 n \rceil+1}$ be the smallest natural number such that $f^j|_{\Delta^*}$ is a diffeomorphism with $f^j(\Delta) \subseteq J_{n - \lceil b_2 n \rceil}(c_i)$, either for $i=0$ or $i=1$.
		We observe that $f^j(\Delta^*)$ contains a $\tau$-scaled neighborhood of $f^j(I \cup I')$: by real bounds and Corollary \ref{coro_twolevelwidehatP}, there exist $\widehat{C_2} = \widehat{C_2}(\widehat{C})>1$ and $\widehat{\mu} \in (0,1)$ such that$$\tau \geq \widehat{C_2} \frac{|f^j(\Delta^*)|}{|f^j(I \cup I')|} \geq \widehat{C_2} \frac{(2 + \widehat{C})|f^j(\Delta)|}{
			|f^j(I\cup I')|} \geq \frac{\widehat{C_2} (2+\widehat{C})}{ \widehat{\mu}^{ \lceil b_2 n \rceil}}.$$
		Therefore,
		\begin{equation*}
				\left( \frac{1}{\tau}+1 \right)^2 \leq 	\left(1+ \frac{\widehat{\mu}^{ \lceil b_2 n \rceil}}{\widehat{C_2}(2 + \widehat{C})} \right)^2
				\leq \exp\left( \frac{4}{\widehat{C_2}(2 + \widehat{C})}
				\widehat{\mu}^{ \lceil b_2 n \rceil} \right),
		\end{equation*}
		and by the Koebe Principle there exists $\widehat{C_3} = \widehat{C_3}(\widehat{C_2}, \widehat{C}) > 1$ such that,
			\[ \sup_{x,y \in I \cup I'} \frac{			|Df^j(x)|}{|Df^j(y)|} \leq \exp(\widehat{C_3} \widehat{\mu}^{ \lceil b_2 n \rceil}).
			\]
		If $x \in I$ and $y \in I'$ are given by the Mean Value Theorem, then we have $\left(\frac{|f^j(I)|}{|I|}\right)/\left(\frac{|f^j(I')|}{|I'|}\right) \leq \exp(\widehat{C_3} \widehat{\mu}^{ \lceil b_2 n \rceil})$, which implies
			\begin{equation}\label{eq:fjvsf0}
				\left| \log \frac{|f^j(I)|}{|f^j(I')|} - \log \frac{|I|}{|I'|} \right| \leq \widehat{C_3}\, \widehat{\mu}^{ \lceil b_2 n \rceil}.
			\end{equation}
	Analogously, replacing $f$ by $g$ we obtain
		\begin{equation}\label{eq:gjvsg0}
			\left| \log \frac{|g^j(h(I))|}{|g^j(h(I'))|} - \log \frac{|h(I)|}{|h(I')|} \right| \leq \widehat{C_3}\, \widehat{\mu}^{ \lceil b_2 n \rceil}.
		\end{equation}
	By the triangle inequality combined with inequalities \eqref{eq:fjvsf0}, \eqref{eq:gjvsg0} (applied to $I$ and $I'$), and Inequality \eqref{eq:dadaporpro_pasodos} (applied to $f^j(I)$ and $f^{j}(I')$) we get
		\begin{align*}
		&\left| \log \frac{|h(I)|}{|I|} - \log \frac{|h(I')|}{|I'|} \right|\\
		&\leq \left| \log \frac{|g^j(h(I))|}{|g^j(h(I'))|} - \log \frac{|h(I)|}{|h(I')|} \right|
		+ 	\left| \log \frac{|h(f^j(I))|}{|h(f^j(I'))|} - \log \frac{|f^j(I)|}{|f^j(I')|} \right| +	\left| \log \frac{|f^j(I)|}{|f^j(I')|} - \log \frac{|I|}{|I'|} \right| \\
		&\leq  2\widehat{C_3}\, \widehat{\mu}^{ \lceil b_2 n \rceil} + C_6\, \lambda_5^n\\
		& \leq \widehat{C_4}\, \mu_3^{n},
	\end{align*}
for $\mu_3=\max\{ \lambda_3,\widehat{\mu}^{b_2/(1+b_2)}\}$.
}
\case{$I$ and $I'$ are not adjacent but they  contained in the same interval of $\widehat{\mathcal{P}}_{n}$.}{Let $\Delta \in \widehat{\mathcal{P}}_{n - \lceil b_2 n \rceil}$ be the interval that contains $I$ and $I'$, let $j<q_{n - \lceil b_2 n \rceil+1}$ be the smallest natural number such that $f^j|_{\Delta^*}$ is a diffeomorphism with $f^j(\Delta) \subseteq J_{n - \lceil b_2 n \rceil}(c_i)$, either for $i=0$ or $i=1$.
	The interval $f^j(\Delta^*)$ contains a $\tau$-scaled neighborhood of $f^j(I)$ and of $f^j(I')$, where $\tau \geq \frac{\widehat{C_2} (2+\widehat{C})}{ \widehat{\mu}^{ \lceil b_2 n \rceil}}$. By Koebe Principle, 
	$\sup_{x,y \in I } |Df^j(x)|/|Df^j(y)| \leq \exp(\widehat{C_3} \widehat{\mu}^{ \lceil b_2 n \rceil}).$
	Since  
	\[
	\sup_{\substack{x,y \in I,\\ |Df^j(y)|=1}} |Df^j(x)| \leq \sup_{x,y \in I } \frac{|Df^j(x)|}{|Df^j(y)|},
	\]
	then if $x \in I$ is given by Mean Value Theorem we obtain
	\begin{equation*}
		\left| \log \frac{|f^j(I)|}{|I|} \right| \leq \widehat{C_3}\, \widehat{\mu}^{ \lceil b_2 n \rceil}.
	\end{equation*}
	Analogously, for $x \in I'$  given by Mean Value Theorem we have,
	\begin{equation*}
		\left| \log \frac{|f^j(I')|}{|I'|} \right| \leq \widehat{C_3}\, \widehat{\mu}^{ \lceil b_2 n \rceil}.
	\end{equation*}
	A combination of the two last inequalities gives us Inequality \eqref{eq:fjvsf0} above. The proof follows after a similar argument as in the end of Case 1.
}
\case{$I$ and $I'$ are adjacent but are not contained in the same interval of $\widehat{\mathcal{P}}_{n}$.}{Let $\Delta, \Delta' \in \widehat{\mathcal{P}}_{n - \lceil b_2 n \rceil}$ be the intervals that contains $I$ and $I'$, respectively. Let $j<q_{n - \lceil b_2 n \rceil+1}$ and $i<q_{n - \lceil b_2 n \rceil+1}$ be the smallest natural numbers such that $f^j|_{\Delta^*}$ is a diffeomorphism with $f^j(\Delta) \subseteq J_{n - \lceil b_2 n \rceil}(c_i)$ and $f^i|_{\Delta'^*}$ is a diffeomorphism with $f^i(\Delta') \subseteq J_{n - \lceil b_2 n \rceil}(c_i)$, either for $i=0$ or $i=1$. As above, $f^j(\Delta^*)$ and $f^i(\Delta'^*)$  contains a $\tau$-scaled neighborhood of $f^j(I)$ and of $f^i(I')$, respectively,  where $\tau \geq \frac{\widehat{C_2} (2+\widehat{C})}{ \widehat{\mu}^{ \lceil b_2 n \rceil}}$. Since 
	\[
	\sup_{\substack{x,y \in I,\\ |Df^j(y)|=1}} |Df^j(x)| \leq \sup_{x,y \in I } \frac{|Df^j(x)|}{|Df^j(y)|} \ \text{and} \ 
	\sup_{\substack{u,v \in I',\\ |Df^i(v)|=1}} |Df^i(u)| \leq \sup_{u,v \in I'} \frac{|Df^i(u)|}{|Df^i(v)|},
	\]
	then by Koebe distortion principle and for $x \in I$ and $u \in I'$ given by Mean Value Theorem we obtain
	\begin{equation*}
		\left| \log \frac{|f^j(I)|}{|I|} \right| \leq \widehat{C_3}\, \widehat{\mu}^{ \lceil b_2 n \rceil} \ \text{and} \
		\left| \log \frac{|f^i(I')|}{|I'|} \right| \leq \widehat{C_3}\, \widehat{\mu}^{ \lceil b_2 n \rceil}.
	\end{equation*}
	Analogously,  we obtain
	\begin{equation*}
		\left| \log \frac{|g^j(h(I))|}{|h(I)|} \right| \leq \widehat{C_3}\, \widehat{\mu}^{ \lceil b_2 n \rceil} \ \text{and} \
		\left| \log \frac{|g^i(h(I'))|}{|h(I')|} \right| \leq \widehat{C_3}\, \widehat{\mu}^{ \lceil b_2 n \rceil}.
	\end{equation*}
Therefore,
	\begin{align*}
	\left| \log \frac{|h(I)|}{|I|} \right|	\leq 
	\left| \log \frac{|h(f^j(I))|}{|f^j(I)|} \right| + \left| \log \frac{|f^j(I)|}{|I|} \right| + 	\left| \log \frac{|h(I)|}{|h(f^j(I))|} \right| 
%	&= \left| \log \frac{|h(f^j(I))|}{|f^j(I)|} \right| + \left| \log \frac{|f^j(I)|}{|I|} \right| + 	\left| \log \frac{|h(I)|}{|g^j(h(I))|} \right| \\
	\leq 2C_6\, \lambda_5^n + \widehat{C_3}\, \widehat{\mu}^{ \lceil b_2 n \rceil} \leq C_7\, \mu_3^{n}.
\end{align*}
Since we also get an equivalent inequality for $I'$, 
%\begin{align*}
%	\left| \log \frac{|h(I')|}{|I'|} \right| \leq C_6\, \mu_3^{n},
%\end{align*}
then
\begin{align*}
	\left| \log \frac{|h(I)|}{|I|} - \log \frac{|h(I')|}{|I'|}  \right|	 \leq C_7\, \mu_3^{n}.
\end{align*}
}
	\end{caseof}

	\section{Some comments about the multi-critical case}\label{sec:commentsmcm}
We note that the criterion for smoothness given by Proposition \ref{criterion} is quite general, and therefore it is independent of the quantity of critical points.  Let us also recall that Proposition \ref{criterion} is true even if adjacent intervals of the circle partition are not comparable (see also \cite[Remark 4]{KT}). Moreover, the construction of the two-bridges partition can be extend to the multi-critical setting: if the pre-images of the critical points inside $I_n(c_0)$ are all well separated, that means that the exists an interval of the form $[x,f^{q_{n+1}}(x)]$ between them, then the construction is very similar. If some of those pre-images are arbitrarily close to each other (inside or not of the same dynamical interval), we define those pre-images as endpoints of the multi-bridge partition (the endpoints of the classical dynamical interval that contains them will not be endpoints of the multi-bridge partition). In the last case we will have at least one interval in the multi-bridge dynamical partition which has arbitrarily small length. In this case, the multi-bridge partition will not satisfy the real bounds property (Theorem \ref{teobeau}). Despite of that, we can obtain the same estimates given by  Proposition \ref{pro_pasouno} for endpoints  in the multi-bridge partition. However, we do not know if the estimates given by propositions \ref{Pro_pasodos}, \ref{Pro_pasotres} and Lemma \ref{lema_pasocuatro}  are  true for intervals of arbitrarily small length  in the multi-bridge partition. Proposition  \ref{Pro_pasocuatro} is a consequence of the results before.	 
	
	\subsection*{Acknowledgments} The author would like to thank Michael Yampolsky for asking the question that encouraged this paper and also to Konstantine Khanin, Corinna Ulcigrai and Frank Trujillo for useful discussions.

\end{document}